\documentclass{article}

\usepackage{microtype}
\usepackage{graphicx}
\usepackage{subfigure}
\usepackage{booktabs} % for professional tables

\usepackage{hyperref}

\usepackage[accepted]{icml2019}

\usepackage{pifont}

\usepackage{epsfig}
\usepackage{amssymb}
\usepackage{amsmath}
\usepackage{amsthm}
\usepackage{amsfonts}
\usepackage{bbding}
\usepackage{array}
\usepackage{caption,tabularx,booktabs}
\usepackage{subfigure}
\usepackage{arydshln}
\usepackage{xargs}                      % Use more than one optional parameter in a new commands
\usepackage[pdftex,dvipsnames]{xcolor}  % Coloured text etc.

\usepackage[colorinlistoftodos,prependcaption,textsize=tiny]{todonotes}
\newcommandx{\unsure}[2][1=]{\todo[inline,linecolor=red,backgroundcolor=red!25,bordercolor=red,#1]{#2}}
\newcommandx{\change}[2][1=]{\todo[linecolor=blue,backgroundcolor=blue!25,bordercolor=blue,#1]{#2}}
\newcommandx{\info}[2][1=]{\todo[linecolor=OliveGreen,inline,backgroundcolor=OliveGreen!25,bordercolor=OliveGreen,#1]{#2}}
\newcommandx{\improvement}[2][1=]{\todo[linecolor=Plum,inline,backgroundcolor=Plum!25,bordercolor=Plum,#1]{#2}}

\usepackage[font=small,skip=5pt]{caption}

\newtheorem{thm}{Theorem}
\newtheorem{lem}{Lemma}

\newtheorem{cor}{Corollary}

\newtheorem{defn}{Definition}

\newtheorem{ass}{Assumption}

\newcolumntype{C}[1]{>{\centering\let\newline\\\arraybackslash\hspace{0pt}}m{#1}}

\newcommand\tagthis{\addtocounter{equation}{1}\tag{\theequation}}

\DeclareMathOperator{\Ocal}{\mathcal{O}}

\renewcommand{\top}{T}

\makeatletter
\newcounter{subthm} 
\let\savedc@thm\c@hyp

\newcommand{\normhyp}{%
  \let\c@hyp\savedc@hyp % revert to the old one
  \renewcommand\thehyp{\arabic{hyp}}%
} 
\makeatother

\makeatletter
\newcounter{subass} 
\let\savedc@ass\c@hyp

\makeatother

\icmltitlerunning{Characterization of Convex Objective Functions and Optimal Expected Convergence Rates for SGD}

\begin{document}

\twocolumn[
\icmltitle{Characterization of Convex Objective Functions and Optimal Expected Convergence Rates for SGD}

% It is OKAY to include author information, even for blind
% submissions: the style file will automatically remove it for you
% unless you've provided the [accepted] option to the icml2019
% package.

% List of affiliations: The first argument should be a (short)
% identifier you will use later to specify author affiliations
% Academic affiliations should list Department, University, City, Region, Country
% Industry affiliations should list Company, City, Region, Country

% You can specify symbols, otherwise they are numbered in order.
% Ideally, you should not use this facility. Affiliations will be numbered
% in order of appearance and this is the preferred way.
\icmlsetsymbol{equal}{*}

\begin{icmlauthorlist}
\icmlauthor{Marten van Dijk}{af1}
\icmlauthor{Lam M. Nguyen}{af2}
\icmlauthor{Phuong Ha Nguyen}{af1}
\icmlauthor{Dzung T. Phan}{af2}
\end{icmlauthorlist}

\icmlaffiliation{af1}{Department of Electrical and Computer Engineering, University of Connecticut, CT, USA.}
\icmlaffiliation{af2}{IBM Research, Thomas J. Watson Research Center, NY, USA}

\icmlcorrespondingauthor{Marten van Dijk}{marten.van$\_$dijk@uconn.edu}
\icmlcorrespondingauthor{Lam M. Nguyen}{LamNguyen.MLTD@ibm.com}
\icmlcorrespondingauthor{Phuong Ha Nguyen}{phuongha.ntu@gmail.com}
\icmlcorrespondingauthor{Dzung T. Phan}{phandu@us.ibm.com}

% You may provide any keywords that you
% find helpful for describing your paper; these are used to populate
% the "keywords" metadata in the PDF but will not be shown in the document
\icmlkeywords{Machine Learning, ICML}

\vskip 0.3in
]

% this must go after the closing bracket ] following \twocolumn[ ...

% This command actually creates the footnote in the first column
% listing the affiliations and the copyright notice.
% The command takes one argument, which is text to display at the start of the footnote.
% The \icmlEqualContribution command is standard text for equal contribution.
% Remove it (just {}) if you do not need this facility.

%\printAffiliationsAndNotice{}  % leave blank if no need to mention equal contribution
\printAffiliationsAndNotice{\icmlEqualContribution} % otherwise use the standard text.

\begin{abstract}
We study Stochastic Gradient Descent (SGD) with diminishing step sizes for convex objective functions. We introduce a definitional framework and theory that defines and characterizes a core property, called curvature, of convex objective functions. In terms of curvature we can derive a new inequality that can be used to compute an optimal sequence of diminishing step sizes by solving a differential equation. Our exact solutions confirm known results in literature and allows us to fully characterize a new regularizer with its corresponding expected convergence rates.
\end{abstract}

\section{Introduction}
\label{sec:intro}

It is well-known that the following stochastic optimization problem
\begin{align*}
\min_{w \in \mathbb{R}^d} \left\{ F(w) = \mathbb{E} [ f(w;\xi) ] \right\}, \tagthis \label{main_prob_expected_risk}  
\end{align*}
where $\xi$ is a random variable  obeying some distribution can be solved efficiently by stochastic gradient descent (SGD) \cite{RM1951}. The SGD algorithm is described in Algorithm~\ref{sgd_algorithm}.

\begin{algorithm}[h]
   \caption{Stochastic Gradient Descent (SGD) Method}
   \label{sgd_algorithm}
\begin{algorithmic}
   \STATE {\bfseries Initialize:} $w_0$
   \STATE {\bfseries Iterate:}
   \FOR{$t=0,1,2,\dots$}
  \STATE Choose a step size (i.e., learning rate) $\eta_t>0$. 
  \STATE Generate a random variable $\xi_t$.
  \STATE Compute a stochastic gradient $\nabla f(w_{t};\xi_{t}).$
   \STATE Update the new iterate $w_{t+1} = w_{t} - \eta_t \nabla f(w_{t};\xi_{t})$.
   \ENDFOR
\end{algorithmic}
\end{algorithm}

If we define $f_i(w):=f(w;\xi_i)$ for a given training set $\{(x_i,y_i)\}_{i=1}^n$ and  $\xi_i$ is a random variable that is defined by a single random sample  $(x,y)$  pulled uniformly from the training set, then empirical risk minimization reduces to 
\begin{gather}\label{main_prob}
\min_{w \in \mathbb{R}^d} \left\{ F(w) = \frac{1}{n} \sum_{i=1}^n f_i(w) \right\}.  
\end{gather}

Problem \eqref{main_prob}, which can also be solved by gradient descent (GD) \cite{nesterov2004,Nocedal2006NO}, has been discussed in many supervised learning applications \cite{ESL}. As an important note, a class of variance reduction methods \cite{SAG,SAGA,SVRG,Nguyen2017_sarah} has been proposed for solving \eqref{main_prob} in order to reduce the computational cost. Since all these algorithms explicitly use the finite sum form of \eqref{main_prob}, they and GD may not be efficient for very large scale machine learning problems. In addition, variance reduction methods are not applicable to \eqref{main_prob_expected_risk}. Hence, SGD is an important algorithm for very large scale machine learning problems and the problems for which we cannot compute the exact gradient. It is proved that SGD has a sub-linear convergence rate with convergence rate $\Ocal(1/t)$ in the strongly convex cases \cite{bottou2016optimization,Nguyen2018,Gower2019}, and $\Ocal(1/\sqrt{t})$ in the general convex cases \cite{nemirovsky1983problem,nemirovski2009robust}, where $t$ is the number of iterations.

In this paper we derive convergence properties for SGD applied to \eqref{main_prob_expected_risk} for many different flavors of convex  objective functions $F$. We introduce a new notion called $\omega$-convexity where $\omega$ denotes a function with certain properties (see Definition~\ref{def:omega_convex}). % $F$ %is $\omega$-convex (see Definition~\ref{def:omega_convex}). 
Depending on $\omega$, $F$ 
can be convex or strongly convex, or something in between, i.e., $F$ is not strongly convex but is ``better'' than ``plain'' convex. This region between plain convex and strongly convex $F$ will be characterized by a new notion for convex objective functions called curvature (see Definition~\ref{def:curvature}).

Convex and non-convex optimization are well-known problems in the literature (see e.g. \cite{SAGjournal,SAGA,schmidt2013fast,nonconvexSVRG}).
% \cite{SAGjournal,SAGA,nonconvexSVRG,Nguyen2017_sarahnonconvex,schmidt2013fast}. 
The problem in the middle range of convexity and non-convexity called quasi-convexity has been studied and analyzed \cite{hazan2015beyond}. Convex optimization is a basic and well studied primitive in machine learning. In some applications, the optimization problems may be non-strongly convex but may have specific structure of convexity. For example, a classical least squares problem with 
\[
f_i(w) = (a_i^\top w - b_i)^2
\]
is convex for some data parameters $a_i \in \mathbb{R}^d$ and $b_i \in \mathbb{R}$. When an $\ell_2$-norm regularization $\|w\|^2$ is employed (ridge regression), the regularized problem becomes strongly convex. Group sparsity is desired in some domains, one can add an $\ell_{2,1}$ regularization $\sum_i \|w_{[i]}\|$ \cite{WrightSPARSA}. This problem is no longer strongly convex, but it should be ``stronger'' than plain convex. 

To the best of our knowledge, there are no specific results or studies in the middle range of convexity and strong convexity. In this paper, we provide a new definition of convexity and study its convergence analysis. 

In our analysis, the following assumptions are required.\footnote{Here and in the remainder of the paper $\|\cdot\|$ stands for the 2-norm.} 

\begin{ass}[$L$-smooth]
\label{ass_smooth}
$f(w;\xi)$ is $L$-smooth for every realization of $\xi$, i.e., there exists a constant $L > 0$ such that, $\forall w,w' \in \mathbb{R}^d$, 
\begin{align*}
\| \nabla f(w;\xi) - \nabla f(w';\xi) \| \leq L \| w - w' \|. \tagthis\label{eq:Lsmooth_basic}
\end{align*} 
\end{ass}

Assumption \ref{ass_smooth} implies that $F$ is also $L$-smooth. 

\begin{ass}[convex] \label{ass_convex}
$f(w;\xi)$ is convex for every realization of $\xi$, i.e., $\forall w,w' \in \mathbb{R}^d$, 
\begin{gather*}
f(w;\xi)  - f(w';\xi) \geq \langle \nabla f(w';\xi),(w - w') \rangle.
\end{gather*}
\end{ass}

Assumption \ref{ass_convex} implies that $F$ is also convex. 

We assume that $f(w;\xi)$ is $L$-smooth and convex for every realization of $\xi$. Then, according to~\cite{nesterov2004}, for all $w,w' \in \mathbb{R}^d$,
\begin{align*}
\| \nabla f(w;\xi) & - \nabla f(w';\xi) \|^2 \\
& \leq L \langle \nabla f(w;\xi) - \nabla f(w';\xi), w-w'\rangle. \tagthis \label{smooth_convex}
\end{align*}

The requirement of existence of unbiased gradient estimators, i.e., 
$\mathbb{E_\xi}[\nabla f(w;\xi)] = \nabla F(w),$ 
 for any fixed $w$
is in need for applying SGD to the general form \eqref{main_prob_expected_risk}.

\noindent{\bf Contributions and Outline.} 

%\begin{enumerate}
     %\item 
     1- Our convergence analysis of SGD for convex objective functions is based on a new recurrence on the expected convergence rate stated in Lemma \ref{lem:E_Yt} (Sec. \ref{sec:convex_optz}).
    As a side result this recurrence is used to show in Theorem \ref{thm_general_convex_wp1_01} (Sec. \ref{sec:convex_optz}) that, for convex objective functions, SGD converges with probability 1 (almost surely) to a global minimum (if one exists). The w.p.1 result is an adaptation of the w.p.1 result in \cite{Nguyen2018} for the strongly convex case.
    
    %\item 
    2- We introduce a new framework and define $\omega$-convex objective functions in Definition~\ref{def:omega_convex} (Sec. \ref{sec:flavors}) and the curvature of convex objective functions in Definition~\ref{def:curvature} (Sec. \ref{sec:flavors}). We show how strongly convex and ``plain'' convex objective functions fit this picture, as extremes on either end (curvature $1$ and $0$, respectively).
    
    %\item 
    3- In Theorem \ref{th:example} we introduce a new regularizer $G(w)$, for $w\in \mathbb{R}^d$, with curvature $1/2$. It penalizes small $\|w\|$ much less than the 2-norm $\|w\|^2$ regularizer and it penalizes large $\|w\|$ much more than the 2-norm $\|w\|^2$ regularizer. This allows us to enforce more tight control on the size of $w$ %as we are interested in 
    when minimizing a convex objective function.
    
    %\item 
    4- By using the recurrence of Lemma \ref{lem:E_Yt} (Sec. \ref{sec:convex_optz}) and a new inequality for $\omega$-convex objective functions, we are able to analyze the expected convergence rate of SGD in Sec. \ref{sec:expected}. We characterize the expected convergence rate as a solution to a differential equation. Our analysis matches existing theory; for strongly convex $F$ we obtain a 2-approximate optimal solution and for ``plain'' convex $F$ with no curvature we obtain an optimal step size of order $O(t^{-1/2})$. For the new regularizer we get a precise expression for the optimal step size and expected convergence  rates.
    % \item Experiments verify our theory (Sec. \ref{sec:exp}).
%\end{enumerate}

\section{Convex Optimization}
\label{sec:convex_optz}

In convex optimization we only assume that $f(w;\xi)$ is $L$-smooth and convex for every realization of $\xi$. Under these assumptions, the objective function
$F(w) = \mathbb{E}_{\xi}[f(w;\xi)]$
is also $L$-smooth and convex. However, the assumptions are too weak to guarantee a unique global minimum for $F(w)$. For this reason we introduce
$$\mathcal{W}^*=\{ w_* \in \mathbb{R}^d \ : \ \forall_{w\in \mathbb{R}^d} \ F(w_*)\leq F(w)\} $$ as the set of all $w_*$ that minimize $F(.)$. The set $\mathcal{W}^*$ may be empty implying that there does not exist a global minimum. If $\mathcal{W}^*$ is not empty, it may contain many vectors $w_*$ implying that a global minimum exists but that it is not unique.

\begin{ass}[global minimum exists] \label{ass_global}
Objective function $F$ has a global minimum.
\end{ass}

This assumption implies that
\begin{gather*}
    \forall_{w_* \in \mathcal{W}^*} \ \nabla F(w_*)=0 \ \mbox{ and } \\
    \exists_{F_{min}} \forall_{w_* \in \mathcal{W}^*} \ F(w_*)=F_{min}.
\end{gather*}
% $$ \forall_{w_* \in \mathcal{W}^*} \ \nabla F(w_*)=0 \ \mbox{ and } \ \exists_{F_{min}} \forall_{w_* \in \mathcal{W}^*} \ F(w_*)=F_{min}. $$ 

%\textcolor{red}{Lam: We may need to assume here that there exists $F_{min}$. The reason is that $F$ is linear, which is smooth and convex, but may not exist lower bound. But the existence of $F_{min}$ is a reasonable assumption, otherwise there is nothing to find.}

With respect to $\mathcal{W}^*$ we define
$$N=\underset{w_* \in \mathcal{W}^*}{\mbox{sup}}\mathbb{E}_\xi[\|\nabla f(w_*;\xi) \|^2]. $$
%We assume $N$ is finite.
%\begin{align*}
%	N = \text{sup}_{w_* \in \mathcal{W}^*} \mathbb{E}_\xi[\|\nabla f(w_*;\xi) \|^2].
%\end{align*}

\begin{ass}[finite $N$] \label{ass_N} We assume $N$ is finite.
%\begin{eqnarray*}
% &&   \underset{w_* \in \mathcal{W}^*}{\mathrm{sup}}\mathbb{E}_\xi[\|\nabla f(w_*;\xi) \|^2]<\infty, 
%\end{eqnarray*}
%where $\mathcal{W}^*=\{ w_* \in \mathbb{R}^d \ : \ \forall_{w\in \mathbb{R}^d} \ F(w_*)\leq F(w)\}$.    
\end{ass}

Without explicitly stating, each of the lemmas and theorems in the remainder of this paper assume Assumptions \ref{ass_smooth}, \ref{ass_convex}, \ref{ass_global}, and \ref{ass_N}.

For the recursively computed values $w_t$, we define
$$Y_t = \underset{w_* \in \mathcal{W}^*}{\mbox{inf}} \|w_t - w_*\|^2 \mbox{ and } E_t=F(w_t)-F(w_*).$$
These quantities measure the convergence rate towards one of the global minima.

\begin{lem}
\label{lem:E_Yt} 
Let $\mathcal{F}_t$ be a $\sigma$-algebra which contains $w_0,\xi_0,w_1,\xi_1,\dotsc,w_{t-1},\xi_{t-1},w_t$. Assume $\eta_t\leq 1/L$. For any given $w_* \in \mathcal{W}^*$, we have 
\begin{equation}
\label{eq:xy00}
 \mathbb{E}[Y_{t+1}| \mathcal{F}_t] \leq \mathbb{E}[Y_{t}| \mathcal{F}_t] -2\eta_t(1-\eta_tL)E_t + 2\eta^2_t N. 
\end{equation}
\end{lem}

The proof of Lemma~\ref{lem:E_Yt} is presented in supplemental material~\ref{sec:rec}. Moreover, an immediate application is given by the next theorem (its proof is in supplemental material \ref{sec:wp1}).

\begin{thm} \label{thm_general_convex_wp1_01}
%Suppose that Assumptions \ref{ass_smooth} and \ref{ass_convex} hold. 
Consider SGD with step size sequence such that
\begin{align*}
0 < \eta_t \leq \frac{1}{L} \ , \ \sum_{t=0}^{\infty} \eta_t = \infty \ \text{and} \ \sum_{t=0}^{\infty} \eta_t^2 < \infty. 
\end{align*}
Then, the following holds w.p.1 (almost surely)
\begin{align*}
F(w_t) - F(w_*) \to 0, 
\end{align*}
where $w_{*}$ is any optimal solution of $F(w)$. 
\end{thm}

We note that the convergence w.p.1. in \cite{Nguyen2018} only works in the strongly convex case while our above theorem holds for the case where the objective function is general convex. 

\section{Convex Flavors} \label{sec:flavors}

%Let us define $N=\underset{w_* \in \mathcal{W}^*}{sup}\mathbb{E}_\xi[\|\nabla f(w_*;\xi) \|^2]$ and 

We define functions
\begin{equation}
a(w)=F(w)-F(w_*)=F(w)-F_{min} \label{A}
\end{equation}
and
\begin{equation}
b(w) = \underset{w_* \in \mathcal{W}^*}{\mbox{inf}} \|w-w_* \|^2.\label{B}
\end{equation}
Notice that $a(w)=0$ if and only if $w\in \mathcal{W}^*$ and $b(w)=0$ if and only if $w\in \mathcal{W}^*$.

We introduce a new definition based on $a(w)$ and $b(w)$ which characterizes a multitude of convex flavors of objective functions:

\begin{defn}[$\omega$-convex]
\label{def:omega_convex}
Let $a: \mathbb{R}^d \to [0,\infty)$ and $b: \mathbb{R}^d \to [0,\infty)$ be smooth functions.
Let $\omega: [0,\infty) \to [0,\infty)$ be $\cap$-convex (i.e. $\omega''(\epsilon)<0$) and strictly increasing (i.e., $\omega'(\epsilon)>0$).  Let ${\cal B}\subseteq \mathbb{R}^d$ be a convex set (e.g., a sphere or $\mathbb{R}^d$ itself) such that, first,
$$\omega(a(w))\geq b(w) \mbox{ for all } w \in {\cal B}$$
and, second, $a(w)=0$ implies both $b(w)=0$ and $w\in {\cal B}$.
Then we call the pair of functions $(a,b)$ $\omega$-separable over ${\cal B}$.

If objective function $F$ gives rise to a pair of functions $(a,b)$  as defined by (\ref{A}) and (\ref{B}) which is $\omega$-separable over ${\cal B}$, then we call $F$ $\omega$-convex over ${\cal B}$.
\end{defn}

%\textcolor{red}{Ha to Marten: need to provide more detailed discussion on the ball $\mathcal{B}$}

The objective function being $\omega$-convex is a subcase of the  Error Bound Condition (see Equation (1) in \cite{Bolte2017}) which only requires $\omega$ to be non-decreasing (i.e., $\omega'(\cdot)\geq 0$). 
The Holderian Error Bound (HEB) (also called Local Error Bound, Local Error Bound Condition, or Lojasiewicz Error bound) is a subcase of the Error Bound Condition where $\omega(\epsilon)=c\epsilon^{p}$ where $c>0$ and $p\in (0,2]$ (see Definition 1 of ~\cite{xu2016accelerated}
%or Section 1.2 of ~\cite{Bolte2017} 
where the reader should notice that $b(w)$ in (\ref{B}) represents the {\em squared} Euclidean distance implying that $\omega$ in our notation is the square of the $\omega$ in Equation (6) of \cite{xu2016accelerated}).
%$p/2\in [0,1)$ and not $p\in [0,1)$).
When $p=1$, HEB becomes the
Quadratic Growth Condition~\cite{drusvyatskiy2018error}; in particular, strong convex objective functions satisfy the Quadratic Growth Condition (see also our Lemma \ref{lem:sc}). 
%Values $p\in (1,2]$ represent an even stronger form of convexity. 

%with better asymptotic expected average convergence rate. The interpretation of $p\in (1,2]$ is an even stronger form of convexity t
%

It turns out that our $\omega$-convex notion and HEB are different as they are not a subclass of each other, but they do have an intersection: Notice that for $p\in (1,2]$, $\omega(\epsilon)=c\epsilon^{p}$ is not $\cap$-convex and does not satisfy Definition \ref{def:omega_convex}, hence HEB is not a subclass of $\omega$-convexity. Also $\omega$-convexity is not a subclass of HEB; for example, our special case of $\omega$-convexity as defined in Definition~\ref{def:curvature} and later studied in the rest of the paper is different from HEB (only $r=\infty$ in Lemma ~\ref{lem:asss} and Theorem~\ref{thYE} reflects HEB). HEB and $\omega$-convexity intersect for $\omega(\epsilon)=c\epsilon^{p}$ with $p\in (0,1]$.
The results in this paper imply that $p\in (0,1]$ corresponds to the range of plain convex to strong convex objective functions for which we analyze the expected convergence rates of SGD with optimal step sizes (given the recurrence of Lemma \ref{lem:E_Yt}). To the best of our knowledge there is no existing work on analyzing the convergence of SGD with this $\omega$-convex notion or with HEB. 

%explain that $p=1$ represents strong convex minimization, and we show that $p\in (0,1]$ corresponds to plain convex to strong convex with better asymptotic expected average convergence rate. The interpretation of $p\in (1,2]$ is an even stronger form of convexity than strong convex with better convergence rates for GD and potentially even better convergence rates for SGD (open problem)
%-- this is what is additionally covered by HEB (we are the first to relate this type of condition to SGD). [HEB in this larger range is of interest to the optimization community because of big data problems where faster convergence rate for GD is important  -- In the ML community strong convex implies one minimum and we work with SGD and not interested in this larger range.] Our condition is more general than HEB for $p\in (0,1]$, however, our examples for $r=\infty$ use this type of function

We list a couple of useful insights (proofs are in supplemental material \ref{sec:delta}):

\begin{lem} %Let set ${\cal B}\subseteq \mathbb{R}^d$  and
Let $a: \mathbb{R}^d \to [0,\infty)$ and $b: \mathbb{R}^d \to [0,\infty)$ be smooth functions and let ${\cal B}\subseteq \mathbb{R}^d$ such that $a(w)=0$ implies $b(w)=0$ and $w\in {\cal B}$.

For $\epsilon\geq 0$, we define
$$\delta(\epsilon)=\underset{p:\mathbb{E}_p[a(w)]\leq \epsilon}{\mbox{sup}} \mathbb{E}_p[b(w)],$$
where $p$ represents a probability distribution over $w \in {\cal B}$.

%TO DO Add max 0 part, and below need assumption $a(w)=0$ implies $b(w)=0$. Also mention that this holds for our $(a,b)$. NOPE: Return to convex B with W* part of it and assumption. Then, delta derivation easy. And this is what we use. Remove lemma 5 stuff.

Assuming $\delta(\epsilon)<\infty$ for $\epsilon\geq 0$, $\delta(\cdot)$ is $\cap$-convex and strictly increasing with $\delta(0)=0$. Furthermore,
\begin{enumerate}
    \item The pair of functions $(a,b)$ is $\delta$-separable over ${\cal B}$.
    \item The pair of functions $(a,b)$ is $\omega$-separable over ${\cal B}$ if and only if $\omega(\epsilon)\geq \delta(\epsilon)$ for all $\epsilon\geq 0$.
   % \item There exists a linear function $\omega(x)=\tau + \kappa x$ for some $\tau\geq 0$ and $\kappa>0$ such that $(a,b)$ is $\omega$-separable over ${\cal B}$.
\end{enumerate}
\end{lem}

The lemma shows that $\delta$ is the ``minimal'' function $\omega$ for which $(a,b)$ is $\omega$-separable over ${\cal B}$. 
% NOT THAT IMPORTANT
%Notice that, since $\delta$ is $\cap$-convex and strictly increasing, for all $x>0$, $\delta(x)\leq \delta(\epsilon)+\delta'(\epsilon)(x-\epsilon)$ which represents the tangent line of $\delta(x)$ at $x=\epsilon$. This shows the existence of linear functions $\omega$ under which $(a,b)$ is $\omega$-separable over ${\cal B}$.

The lemma also shows that $a(w)$ and $b(w)$ are not separable over ${\cal B}$ for any function $\omega(\cdot)$ if and only if $\delta(\epsilon)=\infty$ for $\epsilon>0$. This is only possible if ${\cal B}$ is not bounded within some sphere (e.g., ${\cal B}=\mathbb{R}^d$). If ${\cal B}$ is bounded, then there always exists a function $\omega(\cdot)$ such that 
$a(w)$ and $b(w)$ are $\omega$-separable over ${\cal B}$ (e.g., $\omega(x)=\delta(x)$ %or $\omega(x)=\tau + \kappa x$ 
as defined above).

%We notice that if objective function $F$ is convex, then either the corresponding $a(w)$ and $b(w)$ are not separable over ${\cal B}$ ($\delta(\epsilon)=\infty$ for $\epsilon>0$) or $F$ is $\omega$-convex for some function $\omega$ (e.g., $\omega(x)=\delta(x)$ or $\omega(x)=\omega_l(x)=\tau + \kappa x$ as defined above). 

For convex objective functions, we see in practice that the type of distributions $p$ in the definition of $\delta(\cdot)$ can be restricted to having their probability mass within a bounded sphere ${\cal B}$ of $w$ vectors. In the analysis of the convergence rate this corresponds to assuming all $w_t\in {\cal B}$ (see next section).  As discussed above this makes $\delta(\epsilon)$ finite and we are guaranteed to be able to apply the definitional framework as introduced here.
%introduced in the section.
% in such an adapted definition (everywhere in our analysis we restrict $p$ to represent a probability distribution over $w \in {\cal B} \subseteq \mathbb{R}^d$). 
%We argue that in practice we always see that convex objective functions behave like being $\omega$-convex for some $\omega$.  

The relationship towards strongly convex objective functions is given below.

\begin{defn}[$\mu$-strongly convex]
\label{ass_stronglyconvex}
The objective function $F: \mathbb{R}^d \to \mathbb{R}$ is called $\mu$-strongly convex, if for all $w,w' \in \mathbb{R}^d$, 
\begin{gather*}
F(w) - F(w') \geq \langle \nabla F(w'), (w - w') \rangle + \frac{\mu}{2}\|w - w'\|^2. %\tagthis\label{eq:stronglyconvex_00}
\end{gather*}
\end{defn}

For $w'\in \mathcal{W}^*$, $\nabla F(w')=0$. So, for a $\mu$-strongly convex objective function $f$, $F(w)-F(w_*)\geq \frac{\mu}{2}\|w - w_*\|^2$ for all $w_*\in \mathcal{W}^*$ (notice that $\mathcal{W}^*$ has exactly one vector $w_*$ representing the global minimum). This implies that $\frac{2}{\mu}a(w)\geq b(w)$ for $(a,b)$  defined by (\ref{A}) and (\ref{B}):

\begin{lem} \label{lem:sc} If objective function $F$ is $\mu$-strongly convex, then $F$ is $\omega$-convex over $\mathbb{R}^d$ for function $\omega(x)=\frac{2}{\mu}x$. 
\end{lem}

We will show that existing convergence results for strongly convex objective functions can be derived from assuming the weaker $\omega$-convexity property for appropriately selected $\omega$ as given in the above lemma. 

In order to prove bounds on the expected  convergence rate for any $\omega$-convex objective function, we will use the following inequality:

\begin{lem} \label{lem:expineqM}
Let $a: \mathbb{R}^d \to [0,\infty)$ and $b: \mathbb{R}^d \to [0,\infty)$ be smooth functions and assume they are $\omega$-separable over ${\cal B}$ for some $\cap$-convex and strictly increasing function $\omega$ and convex set ${\cal B}\subseteq \mathbb{R}^d$. Let $p$ be a probability distribution over ${\cal B}$. Then, %(1) 
for all $0<x$,
  $$ \frac{\mathbb{E}_p[b(w)]}{\omega'(x)}
\leq (\frac{\omega(x)}{\omega'(x)}-x) + \mathbb{E}_p[a(w)].$$
%A slightly weaker (approximation of this) inequality is, (2)
%for all $0< \alpha\leq x$,
%$$\frac{x\mathbb{E}_p[b(w)]}{\omega(x)-\tau} \leq \frac{2-c_\alpha}{c_\alpha-1}x + \mathbb{E}_p[a(w)],$$
%where
%$$
% c_\alpha= 
% \underset{e\in [\alpha,\infty)}{\mathrm{sup}} \ \
% \underset{x\in [\alpha,e]}{\mathrm{inf}} \frac{\omega(2x)}{\omega(x)} >1.
% $$
\end{lem}

\begin{proof} %We prove (1) and leave the proof of (2) in supplemental material:
Since $\omega(a(w))\geq b(w)$ for all $w\in {\cal B}$, 
$$\mathbb{E}_p[b(w)]\leq \mathbb{E}_p[\omega(a(w))].$$
Since $\omega(\cdot)$ is $\cap$-convex, % and ${\cal B}$ is a convex set,
$$\mathbb{E}_p[\omega(a(w))]\leq \omega(\mathbb{E}_p[a(w)]).$$
Since $\omega$ is $\cap$-convex and strictly increasing, for all $x>0$ and $y>0$, $\omega(y)\leq \omega(x)+\omega'(x)(y-x)$.
Substituting $y=\mathbb{E}_p[a(w)]$ yields
$$\omega(\mathbb{E}_p[a(w)])\leq \omega(x) + \omega'(x)[\mathbb{E}_p[a(w)]-x].$$
Combining the sequence of inequalities, rearranging terms, and dividing by $\omega'(x)$ proves the statement.
\end{proof}

When applying Lemma \ref{lem:expineqM} we will be interested in bounding $\frac{\omega(x)}{\omega'(x)}-x$ from above while maximizing $\frac{1}{\omega'(x)}$. That is, we want to investigate the behavior of
$$ v(\eta) = \mathrm{sup} \{ \frac{1}{\omega'(x)} \ : \ \frac{\omega(x)}{\omega'(x)}-x\leq \eta \}.$$
Notice that the
derivative of $\frac{\omega(x)}{\omega'(x)}-x$ is equal to 
$\frac{-\omega(x)\omega''(x)}{\omega'(x)^2}\geq 0$,
and the derivative $\frac{1}{\omega'(x)}$ is equal to
$\frac{-\omega''(x)}{\omega'(x)^2}\geq 0$.
This implies that $v(\eta)$ is increasing and is alternatively defined as
\begin{equation}
 v(\eta) = \frac{1}{\omega'(x)} \mbox{ where } \eta = \frac{\omega(x)}{\omega'(x)}-x.
    \label{V}
\end{equation}

\begin{cor} \label{corR} Given the conditions in Lemma \ref{lem:expineqM} with $v(\eta)$ defined as in (\ref{V}), for all $0<\eta$,
$$ v(\eta) \mathbb{E}_p[b(w)]
\leq \eta + \mathbb{E}_p[a(w)].$$
\end{cor}

We are able to use this corollary to provide upper bounds on the expected convergence rate if $v(\eta)$ has a ``nice'' form as given in the next definition and lemma.

\begin{defn} \label{def:curvature}
For $h\in (0,1]$, $r>0$, and $\mu>0$, define
$$\omega_{h,r,\mu}(x)
= \left\{ 
\begin{array}{ll}
\frac{2}{\mu h}(x/r)^h, & \mbox{ if } x\leq r, \mbox{ and} \\
\frac{2}{\mu h} + \frac{2}{\mu} ((x/r)-1), & \mbox{ if } x>r.
\end{array}\right.
$$
We say functions $a(w)$ and $b(w)$ are separable by a function with curvature $h\in (0,1]$ over ${\cal B}$ if for some $r,\mu>0$ they are $\omega_{h,r,\mu}$-separable over ${\cal B}$. We define objective function $F$ to have curvature $h\in(0,1]$ over ${\cal B}$ if its associated functions $a(w)$ and $b(w)$ are $\omega_{h,r,\mu}$-separable over ${\cal B}$ for some $r,\mu>0$. 
\end{defn}

The proof of the following lemma is in  supplemental material \ref{sec:expab}. 
 
\begin{lem} \label{veta}
For $v(\eta)$ defined as in (\ref{V}) and $\omega=\omega_{h,r,\mu}$,
$$v(\eta) =\beta h \eta^{1-h} \mbox{ with } \beta = \frac{\mu}{2}  h^{-h}(1-h)^{-(1-h)}  r^h,$$
for $0\leq \eta\leq r$.
\end{lem}

If set ${\cal B}$ is bounded by a sphere, then the supremum $s_a$ and 
$s_b$ of values $a(w)$ and 
$b(w)$, $w\in {\cal B}$, exist (since $a(w)$ and
$b(w)$ are assumed smooth and continuous everywhere). If  $s_b>0$, then trivially, for $h\in (0,1]$,
$$\frac{h\eta}{s_b} \mathbb{E}_p[b(w)]
\leq \eta + \mathbb{E}_p[a(w)].$$
In other words a linear function $v(\eta)=\beta h\cdot \eta$ for some constant $\beta>0$ (e.g., the one of Lemma \ref{veta} for $h\downarrow 0$) does not give any information. %In this case proving results on the expected convergence rate needs a more classical approach.
Nevertheless taking the limit $h\downarrow 0$ will turn out useful in showing that, for convex objective functions with no curvature, a $\eta_t=O(t^{-1/2})$ diminishing step size is optimal  in the sense that the asymptotic behavior of the expected convergence rate cannot be improved.

Concluding the above discussions, convex objective functions can be classified in different convex flavors: either having a curvature $h\in (0,1]$ (where $h=1$ is implied by strong convexity) or having no such curvature. In the latter case we abuse notation and say that the objective function has "curvature $h=0$". With this extended definition, any convex objective function has a curvature $h\in [0,1]$ over ${\cal B}$ and, by Corollary \ref{corR} and Lemma \ref{veta}, there exist  constants $\beta$ and $r$ such that, for $0\leq \eta\leq r$,
 \begin{equation}
     \beta h \eta^{1-h} \mathbb{E}_p[b(w)]
\leq \eta + \mathbb{E}_p[a(w)]
 \end{equation}
 for distributions $p$ over ${\cal B}$. 

%$$ \mathbb{E}_{p_\eta}[a(w)] = 

%(2) $a(w)\geq 0$ and $b(w)\geq 0$ for all $w$, and

%Then $(1/2)^{1/h}\leq \frac{a(w)}{s_a}$ for $b(w)\geq \eta>$ for some $\eta>0$ implies $2s_b (a(w)/s_a)^h\geq b(w)$ for $w\in {\cal B}\setminus \{ w : b(w)\leq \eta\}$.

%TO DO ${\cal B}$ does not have to be convex, it can have a hole: take out zeroes together with an epsilon sliver. Now we can prove the existence of $h$. As long as $w_t$ is not close to the minimum -- a little away -- all the theory is applicable -- wait sliver can be $\eta_t$ in size.

In supplemental material \ref{sec:example} we show the following example which introduces a new regularizer which makes a convex objective function have curvature $h=1/2$ over $\mathbb{R}^d$:

\begin{thm} \label{th:example}
Let 
$$F(w)=H(w) + \lambda G(w)$$
be our objective function where $\lambda>0$,   $H(w)$ is a convex function, and
$$ G(w) = \sum_{i=1}^d [e^{w_i}+e^{-w_i}-2-  w_i^2].$$ 
Then, $F$ is $\omega$-convex over $\mathbb{R}^d$  for $\omega(x)=\frac{2}{\mu h} x^h$ with $h=1/2$ and
$\mu= \frac{\lambda}{9d}$. The associated $v(\eta)$ as defined in (\ref{V}) is equal to
$$v(\eta) =\beta h \eta^{1-h} \mbox{ with } \beta = \frac{\mu}{2}  h^{-h}(1-h)^{-(1-h)}=\mu,$$
for $\mu\geq 0$.
\end{thm}

%It turns out that $G(w)\approx \frac{1}{12}\sum_{i=1}^d w_i^4$ for all $v_i\in [-1,+1]$.
Function $G(w)$ is of interest as it severely penalizes large $|w_i|$ due to the exponent functions, while for small $|w_i|$ the corresponding term in the sum of $G(w)$ is very small (in fact, we subtract $w_i^2$ in order to make it smaller; if we would not have subtracted the $w_i^2$, then $G$ changes into $G(w)+\|w\|^2$ which is strongly convex). This has the possibility to force the global minimum to smaller size when compared to, e.g., $G(w)=\|w\|$ or $G(w)=\|w\|^2$. The price of moving away from using $G(w)=\|w\|^2$ is moving away from having a strong convex objective function, i.e., the curvature over $\mathbb{R}^d$ is reduced from $h=1$ to $h=1/2$. In the next section we show that this leads to a slower expected convergence rate.

\section{Expected Convergence Rate} \label{sec:expected}

We notice that $w_t$ is coming from a distribution determined by the randomness used in the SGD algorithm when computing $w_0,\xi_0,w_1,\xi_1,\dotsc,w_{t-1},\xi_{t-1}$. Let us call this distribution $p^t$. Then, 
% \begin{equation}
% \mathbb{E}[E_t]=\mathbb{E}[F(w_t)-F(w_*)] = \mathbb{E}_{p^t}[F(w)-F(w_*)]= \mathbb{E}_{p^t}[a(w)].  \label{Ea}  
% \end{equation}
\begin{align*}
    \mathbb{E}[E_t] & =\mathbb{E}[F(w_t)-F(w_*)] \\ & = \mathbb{E}_{p^t}[F(w)-F(w_*)] = \mathbb{E}_{p^t}[a(w)]. \tagthis \label{Ea}  
\end{align*}
Since distribution $p^t$ determines $w_t$, we also have
% \begin{equation} \mathbb{E}[Y_t]=\mathbb{E}[\underset{w_* \in \mathcal{W}^*}{\mbox{inf}} \|w_t-w_* \|^2] = \mathbb{E}_{p^t}[\underset{w_* \in \mathcal{W}^*}{\mbox{inf}} \|w-w_* \|^2]
% = 
% \mathbb{E}_{p^t}[b(w)]. \label{Yb}   
% \end{equation}
\begin{align*} 
    \mathbb{E}[Y_t] & =\mathbb{E}[\underset{w_* \in \mathcal{W}^*}{\mbox{inf}} \|w_t-w_* \|^2] \\ &= \mathbb{E}_{p^t}[\underset{w_* \in \mathcal{W}^*}{\mbox{inf}} \|w-w_* \|^2]
= \mathbb{E}_{p^t}[b(w)]. \tagthis \label{Yb}   
\end{align*}

Both $\mathbb{E}_{p^t}[a(w)]$ and $\mathbb{E}_{p^t}[b(w)]$ measure the expected convergence rate. In practice we want to get close to a global minimum and therefore $\mathbb{E}_{p^t}[a(w)]$ is preferred since $a(w_t)=F(w_t)-F(w_*)$. 
%As a side result we will show that for strongly convex objective functions both $\mathbb{E}_{p^t}[a(w)]$ and $\mathbb{E}_{p^t}[b(w)]$ satisfy the same asymptotic behavior in $t$; in this case, $\mathbb{E}_{p^t}[b(w)]$ can be used to measure the expected convergence rate as well.

For $\eta_t\leq \frac{1}{2L}$, Lemma \ref{lem:E_Yt} shows
$$
 \mathbb{E}[Y_{t+1}| \mathcal{F}_t] \leq \mathbb{E}[Y_{t}| \mathcal{F}_t] -\eta_t E_t + 2\eta^2_t N. 
 $$
 After taking the full expectation and rearranging terms this gives
 \begin{equation}
\eta_t \mathbb{E}[E_t]  \leq \mathbb{E}[Y_{t}] - \mathbb{E}[Y_{t+1}]
  + 2\eta^2_t N. \label{eq:Eform}
 \end{equation}
 By assuming $F$ is $\omega$-convex over ${\cal B}$ and $p^t$ has zero probability mass outside ${\cal B}$, application of Lemma \ref{lem:expineqM} and Corollary \ref{corR} after substituting (\ref{Ea}) and (\ref{Yb}) gives
%  $$ \frac{ \mathbb{E}[Y_t]}{\omega'(x)}
%\leq (\frac{\omega(x)}{\omega'(x)}-x) + \mathbb{E}[E_t].$$
 $$ v(\eta) \mathbb{E}[Y_t]
\leq \eta + \mathbb{E}[E_t].$$
 The right hand side can be upper bounded by using (\ref{eq:Eform}):
% \begin{eqnarray*}
%\eta_t \frac{\mathbb{E}[Y_t]}{\omega'(x)}
%&\leq & \eta_t (\frac{\omega(x)}{\omega'(x)}-x) + \eta_t \mathbb{E}[E_t] \\
%&\leq &
%\eta_t (\frac{\omega(x)}{\omega'(x)}-x) +\mathbb{E}[Y_{t}] - \mathbb{E}[Y_{t+1}] + 2\eta^2_t N.
%\end{eqnarray*}
 \begin{eqnarray*}
\eta_t v(\eta) \mathbb{E}[Y_t]
&\leq & \eta_t \eta  + \eta_t \mathbb{E}[E_t] \\
&\leq &
\eta_t \eta  +\mathbb{E}[Y_{t}] - \mathbb{E}[Y_{t+1}] + 2\eta^2_t N.
\end{eqnarray*}
After reordering terms and using $\eta=\eta_t$ we obtain the recurrence:

\begin{lem} \label{lem:conv} If $F$ is $\omega$-convex over ${\cal B}$ and $p^t$ has zero probability mass outside ${\cal B}$ (or equivalently the SGD algorithm never generates a $w_t$ outside ${\cal B}$), then for $\eta_t\leq \frac{1}{2L}$,
%$$
%\mathbb{E}[Y_{t+1}] \leq (1- \frac{\eta_t}{\omega'(x)})\mathbb{E}[Y_{t}] +  (\frac{\omega(x)}{\omega'(x)}-x)\eta_t + 2\eta^2_t N.$$
%$$
\begin{equation}
\mathbb{E}[Y_{t+1}] \leq (1- v(\eta_t) \eta_t)\mathbb{E}[Y_{t}] +   (2N+1) \eta^2_t , \label{eq-recur}
\end{equation}
where $v(\eta_t)$ is defined by (\ref{V}).
\end{lem}

We notice that if the SGD algorithm has proceeded to the $t$-th iteration, then we know that, due to finite step sizes during the iterations so far, the SGD algorithm has only been able to push the starting vector $w_0$ to some $w_t$ within some bounded sphere ${\cal B}_t$ around $w_0$. So, if $F$ is $\omega$-convex over ${\cal B}_i$ for $1\leq i\leq t$, then we may apply the above recurrence up to iteration $t$. Of course, ideally we do not need to assume this and have ${\cal B}=\mathbb{R}^d$ as in Theorem \ref{th:example}.

\begin{ass}[${\cal B}$-bounded] Until sufficient convergence has been achieved, the SGD algorithm never generates a $w_t$ outside ${\cal B}$.
\end{ass}

%In order to solve this recurrence we define a new function
%$$ v(\eta) = \mathrm{sup} \{ \frac{1}{\omega'(x)} \ : \ \frac{\omega(x)}{\omega'(x)}-x\leq \eta \}.$$
%Its definition shows that if $x$ is such that $\frac{\omega(x)}{\omega'(x)}-x\leq \eta_t$, then $\frac{1}{\omega'(x)}\leq v(\eta_t)$. This allows us to rewrite the above recurrence (by using an appropriate $x$) and obtain
%$$
%\mathbb{E}[Y_{t+1}] \leq (1- \eta_t v(\eta_t))\mathbb{E}[Y_{t}] +    (2N+1)\eta^2_t .$$

In  supplemental material \ref{sec:convergence} we prove the following lemmas that solve recurrence (\ref{eq-recur}).

\begin{lem} Suppose that the objective function is $\omega$-convex over ${\cal B}$ and let $v(\eta)$ be defined as in (\ref{V}). Let $n(\cdot)$ be a decreasing step size function representing $n(t)=\eta_t\leq \frac{1}{2L}$. Define
\begin{align*}
    M(t) &=\int_{x=0}^t n(x)v(n(x)) dx \ \text{and} \\
    C(t) &= \exp(-M(t)) \int_{x=0}^{t} \exp(M(x))n(x)^{2}  dx.
\end{align*}
% $$M(y)=\int_{x=0}^y n(x)v(n(x)) dx \mbox{ and }  C(t) = \exp(-M(t)) \int_{x=0}^{t} \exp(M(x))n(x)^{2}  dx .$$
Then recurrence (\ref{eq-recur})
%$$
%\mathbb{E}[Y_{t+1}] \leq (1- \eta_t v(\eta_t))\mathbb{E}[Y_{t}] +  (2N+1)\eta^2_t $$
implies
$$\mathbb{E}[Y_{t}] \leq A \cdot C(t) + B \cdot \exp(-M(t))
$$
for constants 
$$A=(2N+1) \exp(n(0))$$ and $$B=(2N+1)\exp(M(1))n(0)^{2}  + \mathbb{E}[Y_{0}]$$ (they depend on parameter $N$ and starting vector $w_0$).
\end{lem}

\begin{lem}
A close to optimal step size can be computed by solving the differential equation 
$$ \bar{C}(t) = \frac{2 [-\bar{C}'(t)]^{1/2} }{v([-\bar{C}'(t)]^{1/2})}$$
and equating
$$n(t)= [-\bar{C}'(t)]^{1/2}.$$
The solution to the differential equation approaches $C(t)$ for $t$ large enough: For all $t\geq 0$, $C(t)\leq \bar{C}(t)$. For $t$ large enough, $C(t)\geq \bar{C}(t)/2$. 
\end{lem}

\begin{lem} %TO DO REWRITE
\label{lem:asss}
For $v(\eta) =\beta h \eta^{1-h}$
with $h\in (0,1]$, where $\beta>0$ is a constant and $0\leq \eta \leq r$ for some $r\in (0,\infty]$ (including the possibility $r=\infty$),
%with $$\beta = \frac{\mu}{2}  h^{-h}(1-h)^{-(1-h)}  r^h,$$
we obtain
$$\bar{C}(t) = [1/(2-h)]^{h/(2-h)}  (2/\beta)^{2/(2-h)}
%2 [h/(2-h)]^{h/(2-h)} / (\frac{\mu}{2}  h^{-h}(1-h)^{-(1-h)}  r^h)^{2/(2-h)}
(t+\Delta)^{-h/(2-h)} $$
for
$$ n(t) =  \left(\frac{2}{\beta (2-h)}\right)^{1/(2-h)} (t+\Delta)^{-1/(2-h)}
%\sqrt{ [h/(2-h)]^{2/(2-h)} / (\frac{\mu}{2}  h^{-h}(1-h)^{-(1-h)}  r^h)^{2/(2-h)}} t^{-1/(2-h)}.
$$
with 
$$\Delta = \frac{2\max\{2L,1/r\}}{\beta (2-h)}.$$
\end{lem}

The above results show that an objective function with curvature $h=0$ or a very small curvature does not have a fast decreasing expected convergence rate $\mathbb{E}[Y_{t}]$. Nevertheless, the SGD algorithm does not need to converge in $Y_t$. For small curvature the objective function looks very flat and we may still approach $F_{min}$  reasonably fast.

We use the following classical argument: By (\ref{eq:Eform}),
$$\sum_{i=t+1}^{2t} \eta_i \mathbb{E}[E_i] \leq \mathbb{E}[Y_{t+1}]- \mathbb{E}[Y_{2t+1}] + 2N \sum_{i=t+1}^{2t} \eta_i^2.$$
Define the average
$$A_t =\frac{1}{t} \sum_{i=t+1}^{2t} \mathbb{E}[E_i].$$
For $\eta_t=n(t)$ as defined in the previous lemma,
\begin{align*}
    n(2t) t A_t &\leq \sum_{i=t+1}^{2t} \eta_t \mathbb{E}[E_i], \\
    \sum_{i=t+1}^{2t} \eta_i^2 &\leq \int_{x=t}^{2t} n(x)^2 dx =
\int_{x=t}^{2t}  [-\bar{C}'(x)] dx \\
    &= \bar{C}(t) - \bar{C}(2t)
\end{align*}
% $$ n(2t) t A_t \leq \sum_{i=t+1}^{2t} \eta_t \mathbb{E}[E_i],$$
% \begin{eqnarray*} \sum_{i=t+1}^{2t} \eta_i^2 &\leq& \int_{x=t}^{2t} n(x)^2 dx =
% \int_{x=t}^{2t}  [-\bar{C}'(x)] dx \\
% &=& \bar{C}(t) - \bar{C}(2t).
% \end{eqnarray*}
and
\begin{eqnarray*} \mathbb{E}[Y_{t+1}] &\leq&  A\cdot \bar{C}(t+1) + B\cdot \exp(-M(t+1)) \\
&\leq& A\cdot \bar{C}(t) + B\cdot \exp(-M(t)).
\end{eqnarray*}
We derive
\begin{eqnarray*}
M(t)&=& \int_{x=0}^t n(x)v(n(x)) dx
=\beta h\int_{x=0}^t n(x)^{2-h} dx \\
&=& \frac{2h}{2-h} \int_{x=0}^t (t+\Delta)^{-1} dx \\
&=& \frac{2h}{2-h} [ \ln (t+\Delta) - \ln \Delta],
\end{eqnarray*}
hence,
$$ \exp(-M(t)) = (t+\Delta)^{-2h/(2-h)} \Delta^{2h/(2-h)}.$$
Combining all inequalities yields
$$ n(2t) t A_t \leq (2N+A)\cdot \bar{C}(t) + B \Delta^{2h/(2-h)}(t+\Delta)^{-2h/(2-h)}.$$
%
%D\cdot \bar{C}(t),$$
%for constant
%$$D= 2N+A+B\Delta^{2h/(2-h)}.$$
This proves the following theorem: 

\begin{thm} \label{thYE}
For an objective function with curvature $h\in (0,1]$ 
with associated $v(\eta) =\beta h \eta^{1-h}$, where $\beta>0$ is a constant and $0\leq \eta \leq r$ for some $r\in (0,\infty]$ (including the possibility $r=\infty$), a close to optimal step size is 
$$\eta_t =  \left(\frac{2}{\beta (2-h)}\right)^{1/(2-h)} (t+\Delta)^{-1/(2-h)}.$$
%with
%$$\beta = \frac{\mu}{2}  h^{-h}(1-h)^{-(1-h)}  r^h.$$
The corresponding expected convergence rates are 
% \begin{eqnarray*}
% \mathbb{E}[Y_{t}] &\leq& [1/(2-h)]^{h/(2-h)}  (2/\beta)^{2/(2-h)}, \\
% \frac{1}{t} \sum_{i=t+1}^{2t} \mathbb{E}[E_i] &=& [1/(2-h)]^{-(1-h)/(2-h)}  (2/\beta)^{1/(2-h)} t^{ -1/(2-h)}.
% \end{eqnarray*}
% \begin{align*}
% & \mathbb{E}[Y_{t}] \leq [1/(2-h)]^{h/(2-h)}  (2/\beta)^{2/(2-h)}, \\
% & \frac{1}{t} \sum_{i=t+1}^{2t} \mathbb{E}[E_i] = \\ \qquad &= [1/(2-h)]^{-(1-h)/(2-h)}  (2/\beta)^{1/(2-h)} t^{ -1/(2-h)}.
% \end{align*}
\begin{eqnarray*}
\mathbb{E}[Y_{t}] & \leq& A \frac{ [1/(2-h)]^{h/(2-h)}  (2/\beta)^{2/(2-h)} }{(t+\Delta)^{h/(2-h)}} + \\
&& B \frac{\Delta^{2h/(2-h)}}{(t+\Delta)^{2h/(2-h)}},\\
\frac{1}{t} \sum_{i=t+1}^{2t} \mathbb{E}[E_i] &\leq&
(2N+A) A' \frac{(2t+\Delta)^{ 1/(2-h)}}{ (t+\Delta)^{ h/(2-h)}t } + \\
&& B \Delta^{2h/(2-h)}  B'   \frac{(2t+\Delta)^{ 1/(2-h)}}          
{ (t+\Delta)^{2h/(2-h)}t  },
%D \frac{[1/(2-h)]^{-(1-h)/(2-h)}  (2/\beta)^{1/(2-h)}}{ (t+\Delta)^{ 1/(2-h)} }.
\end{eqnarray*}
where $$A'=[1/(2-h)]^{-(1-h)/(2-h)}  (2/\beta)^{1/(2-h)},$$ and $$B'=[1/(2-h)]^{-1/(2-h)}  (2/\beta)^{-1/(2-h)}.$$
\end{thm}

The asymptotic behavior is dominated by the terms with $A$ and $A'$. This shows independence of the expected convergence rates from the starting point $w_0$ since $\mathbb{E}[Y_0]$ only occurs in $B$. We have
$$\mathbb{E}[Y_{t}]= O(t^{-h/(2-h)}) \mbox{ and } \frac{1}{t} \sum_{i=t+1}^{2t} \mathbb{E}[E_i]=O(t^{-1/(2-h)}).$$

For $\mu$-strongly convex objective functions we have $v(\eta)=\frac{\mu}{2}h\eta^{1-h}$ for $h=1$.  Theorem \ref{thYE} (after substituting constants and substituting $r=\infty$) gives, for $A=(2N+1)e^{1/(2L)}$,
$$\mathbb{E}[Y_{t}] \leq \frac{A}{\mu} \left[ \frac{ 16 }{(\mu t+ 8L)}\right] + O(t^{-2})$$
and
$$ \frac{1}{t} \sum_{i=t+1}^{2t} \mathbb{E}[E_i] \leq
\frac{2N+A}{\mu} \left[ \frac{4(2\mu t+8L)}{ (\mu t+8L)t } \right] + O(t^{-2})
$$
for step size
$$ \eta_t = \frac{2}{\mu t/2 + 4L}.$$
In \cite{Nguyen2018}, they report an optimal step size of $2/(\mu t+4L)$, hence, $\eta_{2t}$ is equal to this optimal steps size for the $t$-th iteration and this implies that it takes a factor 2 slower to converge; this is consistent with our derivation in which we use $\bar{C}(t)$ as a $2$-approximate optimal solution.

For the example in Theorem \ref{th:example} with $v(\eta)=\mu h\eta^{1-h}$ for $h=1/2$ (and $r=\infty$), we obtain, for $A=(2N+1)e^{1/(2L)}$,
$$\mathbb{E}[Y_{t}]\leq 
 \frac{A}{\mu} \left[ \frac{ 32 }{3\mu t+8L}\right]^{1/3} + O(t^{-2/3})$$
and
$$\frac{1}{t} \sum_{i=t+1}^{2t} \mathbb{E}[E_i] \leq
\frac{2N+A}{\mu} \left[   \frac{2(6\mu t+8L)^{ 2}}{ (3\mu t+8L)t^3 } \right]^{1/3}+ O(t^{-1}) $$
for step size
$$\eta_t= \left(\frac{2}{3\mu t/2 + 4L}\right)^{2/3}.$$
Due to the smaller curvature we need to choose a larger step size. The expected convergence rates are $O(t^{-1/3})$ and $O(t^{-2/3})$, respectively.

For $h\downarrow 0$, we recognize the classical result which holds for all convex objective functions. In this case the theorem shows that a diminishing step size of $O(t^{-1/2})$ is close to optimal.

%\section{Curvature}

%In our definition and analysis of curvature we will use the functions
%$$\omega_{h,r,\mu}(x)
%= \left\{ 
%\begin{array}{ll}
%\frac{2}{\mu}(x/r)^h, & \mbox{ if } x\leq r, \mbox{ and} \\
%\frac{2}{\mu} + \frac{2}{\mu}h ((x/r)-1), & \mbox{ if } x>r,
%\end{array}\right.
%$$
%for $h\in (0,1]$, $r>0$, and $\mu>0$, and we extend this definition by defining functions
%$$\omega_{0,r,\mu}(x)= \frac{2}{\mu}(1+x/r).$$
%The latter functions correspond to any convex objective function (we know that a convex objective function is $\omega_{0,r,\mu}$-convex for some $r$ and $\mu$).
%Both sets of functions have
%$$v(\eta) =\frac{\mu}{2}  h^{-h}(1-h)^{-(1-h)}  r^h \eta^{1-h}.$$

%explain special class $\omega$ and define curvature, the result on convergence rates, what does this mean for our convexity assumption (strong convexity or convex behavior for far away $w$ -- the behavior inside a sphere of size $r$ matters, even convex function have strong convex behavior outside the sphere; inside the sphere curvature measures how flat it is, the flatter the slower the convergence

\section{Experiments} \label{sec:exp}

%\begin{figure*}[h!t!]%
%\centering
%\subfigure[][]{%
%\label{fig:ex3-a}%
%\includegraphics[height=1.657in]{Figs/Convex.eps}}%
%\hspace{8pt}%
%\subfigure[][]{%
%\label{fig:ex3-b}%
%\includegraphics[height=1.657in]{Figs/OmegaConvex.eps}} \\
%\subfigure[][]{%
%\label{fig:ex3-c}%
%\includegraphics[height=1.657in]{Figs/optimal_h.eps}}%
%\hspace{8pt}%
%\subfigure[][]{%
%\label{fig:ex3-d}%
%\includegraphics[height=1.657in]{Figs/StrongConvex.eps}}%
%\caption[A set of four subfigures.]{Convergence rate for \subref{fig:ex3-a} $f_i(w)$ ($h=0$);
%\subref{fig:ex3-b} $f^{(a)}_i(w)$;
%\subref{fig:ex3-c} the new regularizer $f^{(b)}_i(w)$ ($h=1/2$),
%\subref{fig:ex3-d}  $f^{(c)}_i(w)$ ($h=1$).}%
%\label{fig:ex3}%
%\end{figure*}

\begin{figure*}[h!t!]%
\centering
\subfigure[][]{%
\label{fig:ex3-a}%
\includegraphics[width=0.4\textwidth]{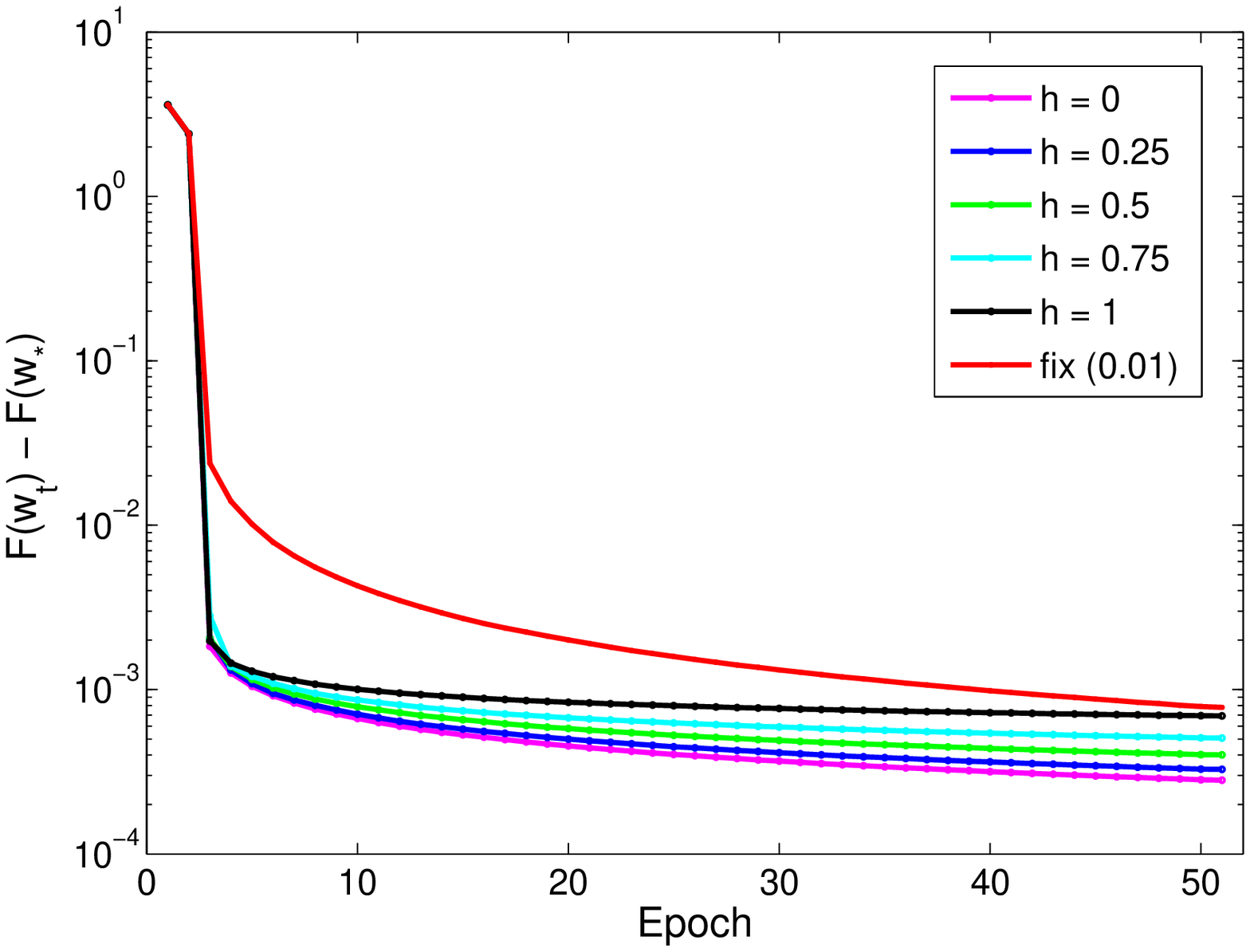}}%
\hspace{8pt}%
\subfigure[][]{%
\label{fig:ex3-b}%
\includegraphics[width=0.4\textwidth]{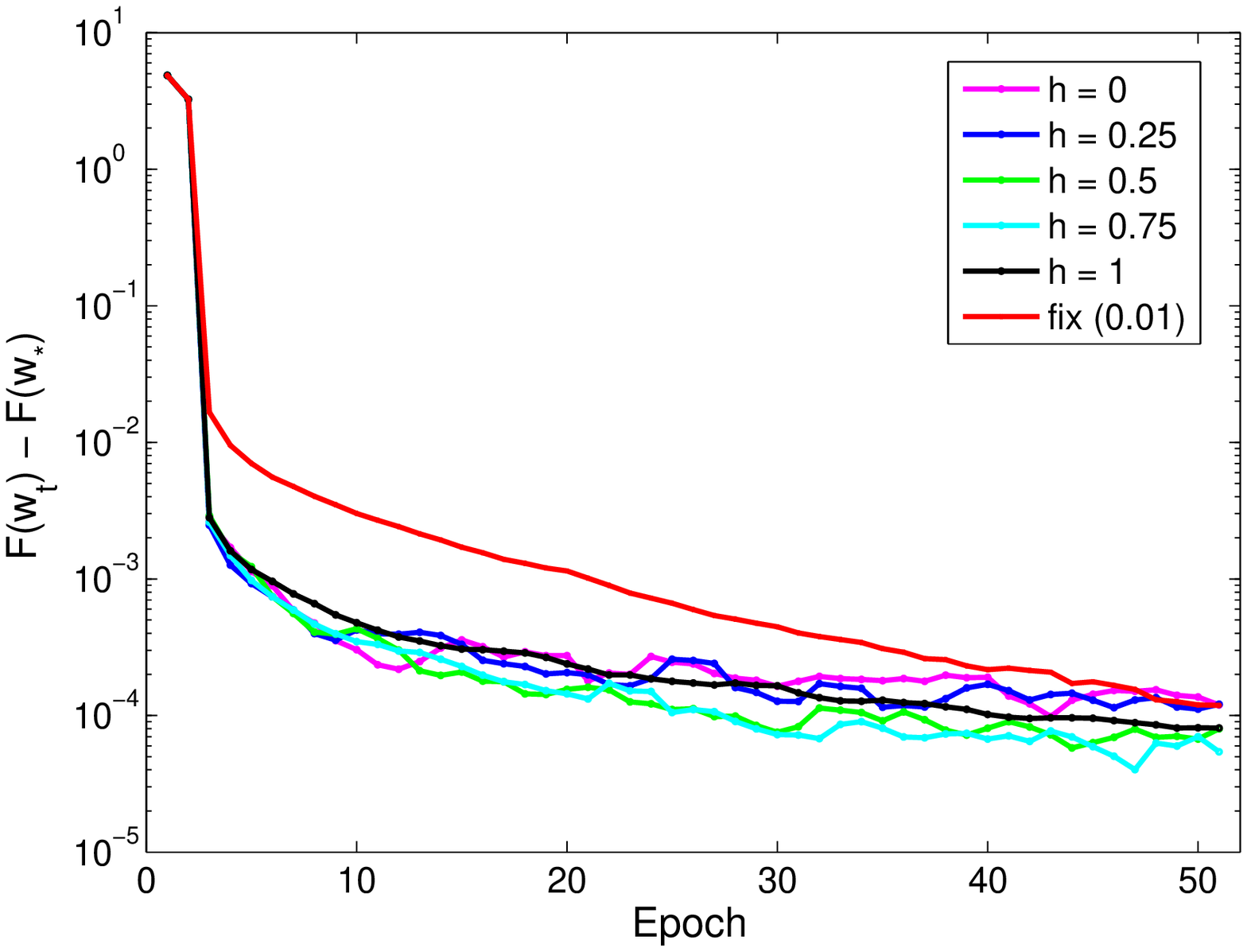}} \\
\subfigure[][]{%
\label{fig:ex3-c}%
\includegraphics[width=0.4\textwidth]{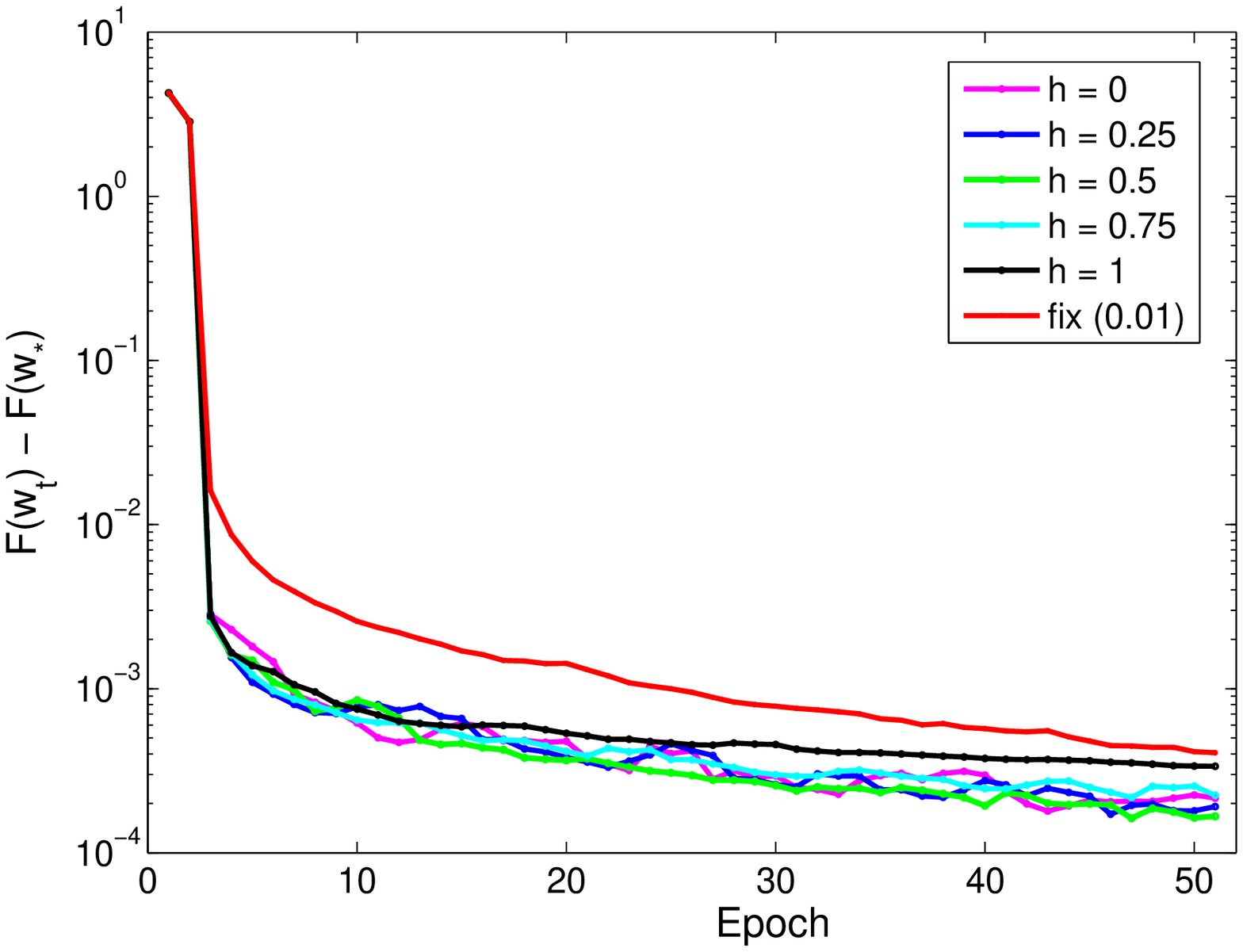}}%
\hspace{8pt}%
\subfigure[][]{%
\label{fig:ex3-d}%
\includegraphics[width=0.4\textwidth]{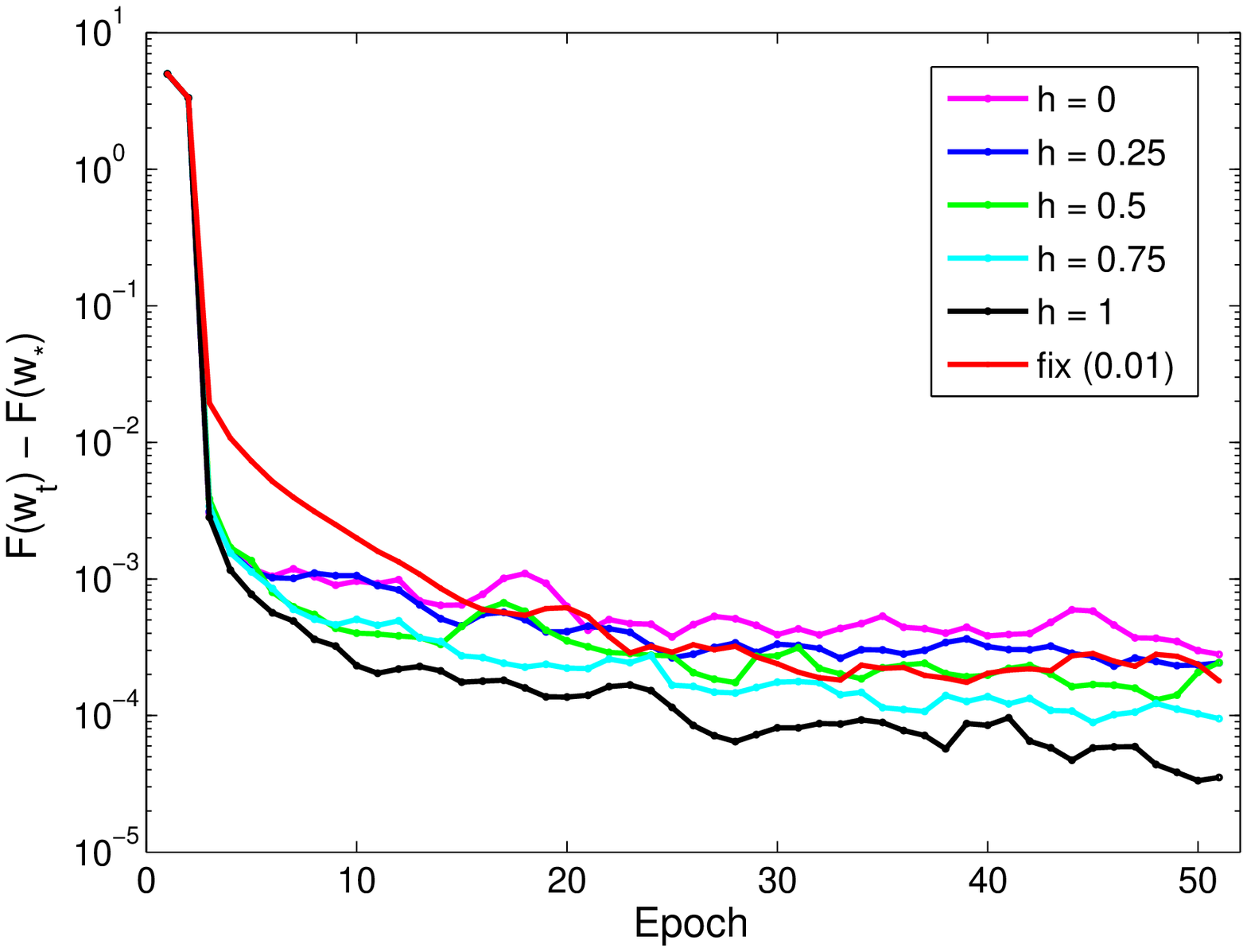}}%
\caption[A set of four subfigures.]{Convergence rate for \subref{fig:ex3-a} $f_i(w)$ ($h=0$);
\subref{fig:ex3-b} $f^{(a)}_i(w)$;
\subref{fig:ex3-c} the new regularizer $f^{(b)}_i(w)$ ($h=1/2$),
\subref{fig:ex3-d}  $f^{(c)}_i(w)$ ($h=1$).}%
\label{fig:ex3}%
\end{figure*}

%\begin{figure}
%    \centering
%    \subfigure[]
%    {
%        \includegraphics[height=1.657in]{Figs/Convex.pdf}
%    }
%    \\
%    \subfigure[]
%    {
%        \includegraphics[height=1.657in]{Figs/OmegaConvex.pdf}
%    }
%    \qquad
%    \subfigure[]
%    {
%        \includegraphics[height=1.657in]{Figs/optimal_h.pdf}
%    }
%    \caption
%    {
%        (a) blah
%        (b) blah
%        (c) blah
%    }
%    \label{fig:foobar}
%\end{figure}

 % \begin{figure*}[h!t!]
    % \centering
    % \begin{subfigure}[b]{0.45\textwidth}
        % \includegraphics[width=\textwidth]{Figs/Convex}
        % \caption{(a)}
        % \label{fig:ex3-a}
    % \end{subfigure}
    % \quad
     % \vspace{0.5cm}
      % %(or a blank line to force the subfigure onto a new line)
    % \begin{subfigure}[b]{0.45\textwidth}
        % \includegraphics[width=\textwidth]{Figs/OmegaConvex}
        % \caption{(b)}
        % \label{fig:ex3-b}
    % \end{subfigure}
    % \begin{subfigure}[b]{0.45\textwidth}
        % \includegraphics[width=\textwidth]{Figs/optimal_h}
        % \caption{(c)}
        % \label{fig:ex3-c}
    % \end{subfigure}
    % \quad
    % \begin{subfigure}[b]{0.45\textwidth}
        % \includegraphics[width=\textwidth]{Figs/StrongConvex}
        % \caption{(d)}
        % \label{fig:ex3-d}
    % \end{subfigure}
% \caption{Convergence rate for \subref{fig:ex3-a} $f_i(w)$ ($h=0$);
% \subref{fig:ex3-b} $f^{(a)}_i(w)$;
% \subref{fig:ex3-c} the new regularizer $f^{(b)}_i(w)$ ($h=1/2$),
% \subref{fig:ex3-d}  $f^{(c)}_i(w)$ ($h=1$).}%
% \label{fig:ex3}%
% \end{figure*}

We consider both unregularized and regularized logistic regression problems with different regularizers to account for convex, $\omega$-convex, and strongly convex cases: 
\begin{eqnarray*}
f_i(w) &=& \log(1 + \exp(-y_i x_i^\top w)) \mbox{ (convex)}\\
f^{(a)}_i(w) &=& f_i(w) +
\lambda \| w \| \mbox{ ($\omega$-convex)}\\
f^{(b)}_i(w) &=& f_i(w) + \lambda G(w) \mbox{ ($\omega$-convex)}\\
f^{(c)}_i(w) &=& f_i(w)  +
\frac{\lambda}{2} \| w\|^2 \mbox{ (strongly convex)},
\end{eqnarray*}
where the penalty parameter $\lambda$ is set to $10^{-3}$. We have not been able to prove the curvature of the objective function $F$ corresponding to $f^{(a)}_i$,  we address this in a general take-away in the conclusion; $F$ corresponding to $f^{(b)}_i$ has curvature $h=1/2$ by Theorem \ref{th:example}; $F$ corresponding to $f^{(c)}_i$ has curvature $h=1$ since it is strongly convex. 

We conducted
experiments on a binary classification dataset  \texttt{mushrooms} from the LIBSVM website\footnote{http://www.csie.ntu.edu.tw/$\sim$cjlin/libsvmtools/datasets/}.
We ran Algorithm
\ref{sgd_algorithm} using the fixed learning rate $\eta = 10^{-2}$ and diminishing step sizes %$\eta_t= \eta/(t+2L\eta)^{1/(2-h)}\approx \eta/t^{1/(2-h)}$ 
$\eta_t = 0.1/t^{1/(2-h)}$ 
for
different values of $h = \{0, 0.25, 0.5, 0.75, 1\}$ to validate theoretical convergence rates given in Theorem \ref{thYE}. For each problem, we experimented with 10 seeds and took the average of function values at the end of each epoch. To smooth out function values due to the ``noise'' from randomness, we reported the moving mean with a sliding window of length 3 for curves in Figure~\ref{fig:ex3}. %\ref{fig:convex}-\ref{fig:StrongConvex}. 

The plots match the theory closely in terms of curvature values and optimal diminishing step sizes. Figure \ref{fig:ex3-a} for convex case with curvature $h=0$ shows the best performance for a step size $\eta_t = 0.1/\sqrt{t}$ corresponding to  $h=0$.  Figure \ref{fig:ex3-b} suggests that  the objective function $F$ corresponding to $f^{(a)}_i$ has curvature close to $h=0.75$; this curvature may be due to convergence to a minimum $w_*$ in a neighborhood where the combination of plain logistic regression and regularizer $\|w\|$ has curvature $0.75$. %in the neighborhood of $w_*$ where convergence takes place. 
In Figure~\ref{fig:ex3-c}, the stepsize rule pertaining to $h=0.5$ yields the top performance for $f_i^{(b)}$ having curvature $h=0.5$. Finally, the strongly convex case $f_i^{(c)}$ having curvature $h=1$, the step size $\eta_t = 0.1/t$, i.e. $h=1$, gives the fastest convergence.

\section{Conclusion}
\label{sec:conclusion}
We have provided a solid framework for analyzing the expected convergence rates of SGD for any convex objective function. Experiments match derived  optimal step sizes. In particular, our new regularizer fits theoretical predictions.

The proposed framework is useful for analyzing any new regularizer, even if theoretical analysis is out-of-scope. One only needs to experimentally discover the curvature $h$ of a new regularizer once. %for {\em one} convex problem. 
After curvature $h$ is determined, the regularizer can be used for any convex problem together with a diminishing step size proportional to the optimal one as given by our theory for curvature $h$. Our theory predicts the resulting expected convergence rates and this can be used together with other properties of regularizers to select the one that best fits a convex problem.

Our framework characterizes a continuum from plain convex to strong convex problems and explains how the expected convergence rates of SGD vary along this continuum. Our metric `curvature' has a one-to-one correspondence to how to choose an optimal diminishing step size  and to the expected and average expected convergence rate. %In this sense our analysis completes our understanding of SGD.

%AAAA 

%Future work: adaptive SGD estimates $h$ or adaptively corrects for this -- will take some fine tuning of the algorithm itself (the filter with which an estimate of "$F(w_t)$" is maintained and feedback loop for estimating $h$ with regularization). A follow-up work ....

%\noindent{\bf Acknowledgement.}

\section*{Acknowledgement}
The authors would like to thank the reviewers for useful suggestions which helped to improve the exposition in the paper. Phuong Ha Nguyen and Marten van Dijk were supported in part by AFOSR MURI under award number FA9550-14-1-0351.

\bibliography{20170724_reference}
\bibliographystyle{icml2019}

%%%%%%%%%%%%%%%%%%%%%%%%%%%%%%%%%%%%%%%%%%%%%%%%%%%%%%%%%%%%%%%%%

% \appendix
\clearpage
\onecolumn
% \appendix

\icmltitle{Characterization of Convex Objective Functions and Optimal Expected Convergence Rates for SGD \\
			Supplementary Material, ICML 2019}

\appendix

% \section*{Appendix}

\section{Recurrence for General Convex Objective Functions} \label{sec:rec}

%\begin{thm}
%\label{lem:E_Yt} 
\textbf{Lemma~\ref{lem:E_Yt}}
\textit{Let $\mathcal{F}_t$ be a $\sigma$-algebra which contains $w_0,\xi_0,w_1,\xi_1,\dotsc,w_{t-1},\xi_{t-1},w_t$. Assume $\eta_t\leq 1/L$. For any given $w_* \in \mathcal{W}^*$, we have} 
\begin{equation}
%\label{eq:xy00}
 \mathbb{E}[Y_{t+1}| \mathcal{F}_t] \leq \mathbb{E}[Y_{t}| \mathcal{F}_t] -2\eta_t(1-\eta_tL)E_t + 2\eta^2_t N. 
\end{equation}
%\end{thm}

\begin{proof} The proof consists of two parts: We first derive a general inequality on $Y_t$ and after this we take the expectation leading to the final result.

We remind the reader that $w_{t+1}=w_t -\eta_t \nabla f(w_t;\xi_t)$ from which we derive, 
for any given $w_* \in \mathcal{W}^*$, 
% \begin{eqnarray*}
% && \|w_{t+1} -w_* \|^2 = \|w_t -w_* -\eta_t \nabla f(w_t;\xi_t) \|^2 \\
% &=& \|w_t -w_* \|^2 - 2\eta_t \langle \nabla f(w_t;\xi_t), w_t - w_* \rangle + \eta^2_t \|\nabla f(w_t;\xi_t) \|^2\\
% &\leq& \|w_t -w_* \|^2 - 2\eta_t \langle \nabla f(w_t;\xi_t), w_t - w_* \rangle  \\
% && \hspace{1cm} + 2\eta^2_t \|\nabla f(w_t;\xi_t)-\nabla f(w_*;\xi_t)\|^2 + 2\eta^2_t \|\nabla f(w_*;\xi_t)\|^2,
% \end{eqnarray*}
\begin{align*}
 \|w_{t+1} -w_* \|^2 &= \|w_t -w_* -\eta_t \nabla f(w_t;\xi_t) \|^2 \\
&= \|w_t -w_* \|^2 - 2\eta_t \langle \nabla f(w_t;\xi_t), w_t - w_* \rangle \\ & \qquad + \eta^2_t \|\nabla f(w_t;\xi_t) \|^2\\
&\leq \|w_t -w_* \|^2 - 2\eta_t \langle \nabla f(w_t;\xi_t), w_t - w_* \rangle  \\
& \qquad + 2\eta^2_t \|\nabla f(w_t;\xi_t)-\nabla f(w_*;\xi_t)\|^2 \\ & \qquad + 2\eta^2_t \|\nabla f(w_*;\xi_t)\|^2,
\end{align*}
where the inequality follows from the general inequality $\|x\|^2\leq 2\|x-y\|^2 + 2\|y\|^2$.

Application of inequality (\ref{smooth_convex}), i.e., $\|\nabla f(w_t;\xi_t)-\nabla f(w_*;\xi_t)\|^2 \leq L \langle \nabla f(w_t;\xi_t)-\nabla f(w_*;\xi_t), w_t -w_*\rangle$, gives 
% \begin{eqnarray*}
% && \|w_{t+1} -w_* \|^2 \\
% &\leq&  \|w_t -w_* \|^2 - 2\eta_t \langle \nabla f(w_t;\xi_t), w_t - w_* \rangle  \\
% && \hspace{1cm} + 2\eta^2_tL \langle \nabla f(w_t;\xi_t)-\nabla f(w_*;\xi_t), w_t -w_*\rangle + 2\eta^2_t \|\nabla f(w_*;\xi_t)|^2
%  \\
% &=& \|w_t -w_* \|^2 - 2\eta_t(1-\eta_tL) \langle \nabla f(w_t;\xi_t), w_t - w_* \rangle \\
% && \hspace{1cm} + 2\eta^2_t \|\nabla f(w_*;\xi_t)\|^2  - 2\eta^2_t L \langle \nabla f(w_*;\xi_t),w_t-w_* \rangle.
% \end{eqnarray*}
\begin{align*}
& \|w_{t+1} -w_* \|^2 \\
&\leq  \|w_t -w_* \|^2 - 2\eta_t \langle \nabla f(w_t;\xi_t), w_t - w_* \rangle  \\
& \qquad + 2\eta^2_tL \langle \nabla f(w_t;\xi_t)-\nabla f(w_*;\xi_t), w_t -w_*\rangle \\ & \qquad + 2\eta^2_t \|\nabla f(w_*;\xi_t)|^2
 \\
&= \|w_t -w_* \|^2 - 2\eta_t(1-\eta_tL) \langle \nabla f(w_t;\xi_t), w_t - w_* \rangle \\
& \qquad + 2\eta^2_t \|\nabla f(w_*;\xi_t)\|^2  - 2\eta^2_t L \langle \nabla f(w_*;\xi_t),w_t-w_* \rangle.
\end{align*}

The convexity assumption states $\langle \nabla f(w_t;\xi_t), w_t - w_* \rangle \geq f(w_t;\xi_t)-f(w_*;\xi_t)$ and this allows us to further develop our derivation: For $\eta_t\leq 1/L$,
% \begin{eqnarray*}
% &&
% \|w_{t+1} -w_* \|^2   \\
% &\leq& \|w_t -w_* \|^2 - 2\eta_t(1-\eta_tL) [f(w_t;\xi_t)-f(w_*;\xi_t)] \\
% && \hspace{1cm} + 2\eta^2_t \|\nabla f(w_*;\xi_t)\|^2  - 2\eta^2_t L \langle \nabla f(w_*;\xi_t),w_t-w_* \rangle.
% \end{eqnarray*}
\begin{align*}
&
\|w_{t+1} -w_* \|^2   \\
&\leq \|w_t -w_* \|^2 - 2\eta_t(1-\eta_tL) [f(w_t;\xi_t)-f(w_*;\xi_t)] \\
& \qquad + 2\eta^2_t \|\nabla f(w_*;\xi_t)\|^2  - 2\eta^2_t L \langle \nabla f(w_*;\xi_t),w_t-w_* \rangle.
\end{align*}

Let $w^t_*$ be such that $\|w_t -w^t_* \|^2=\underset{w_* \in \mathcal{W}^*}{\mbox{inf}}\|w_t -w_* \|^2=Y_t$ (here, we assume for simplicity that the infinum can be realized for some $w^t_*$ and we note that the argument below can be made general with small adaptations). Notice that
$$Y_{t+1}=\underset{w_* \in \mathcal{W}^*}{\mbox{inf}}\|w_{t+1} -w_* \|^2 \leq \| w_{t+1}-w^t_*\|^2.$$
By using $w_*=w^t_*$ in the previous derivation, we obtain
%\begin{equation}
%\label{eq:xxx100}
\begin{align*}
    Y_{t+1} &\leq Y_t  - 2\eta_t(1-\eta_tL) [f(w_t;\xi_t)-f(w^t_*;\xi_t)] \\
    & \qquad + 2\eta^2_t \|\nabla f(w^t_*;\xi_t)\|^2 \\ & \qquad - 2\eta^2_t L \langle \nabla f(w^t_*;\xi_t),w_t-w^t_* \rangle.
\end{align*}
% $$
% Y_{t+1}\leq Y_t  - 2\eta_t(1-\eta_tL) [f(w_t;\xi_t)-f(w^t_*;\xi_t)] + 2\eta^2_t \|\nabla f(w^t_*;\xi_t)\|^2  - 2\eta^2_t L \langle \nabla f(w^t_*;\xi_t),w_t-w^t_* \rangle.
% $$
%\end{equation}

Now, we take the expectation %of (\ref{eq:xy00}) in Theorem \ref{lem:E_Yt} 
with respect to $\mathcal{F}_t$. Notice that
$F(w_t) = \mathbb{E}_{\xi}[f(w_t;\xi)]  = \mathbb{E}_{\xi}[f(w_t;\xi)| \mathcal{F}_t]$
and
$F(w^t_*) = \mathbb{E}_{\xi}[f(w^t_*;\xi)]= \mathbb{E}_{\xi}[f(w^t_*;\xi)| \mathcal{F}_t]$.
This yields
% \begin{eqnarray*}
%  \mathbb{E}[Y_{t+1}| \mathcal{F}_t] &\leq&  \mathbb{E}[Y_{t}| \mathcal{F}_t] -2\eta_t(1-\eta_tL)[F(w_t)-F(w^t_*)] + 2\eta^2_t \mathbb{E}[\|\nabla f(w^t_*;\xi_t) \|^2| \mathcal{F}_t] \\
% && \hspace{1cm} -2\eta^2_t L \langle \nabla F(w^t_*),w_t-w_* \rangle.
% \end{eqnarray*}
\begin{align*}
 \mathbb{E}[Y_{t+1}| \mathcal{F}_t] &\leq  \mathbb{E}[Y_{t}| \mathcal{F}_t] -2\eta_t(1-\eta_tL) [F(w_t)-F(w^t_*)] \\ &\qquad + 2\eta^2_t \mathbb{E}[\|\nabla f(w^t_*;\xi_t) \|^2| \mathcal{F}_t] \\
& \qquad -2\eta^2_t L \langle \nabla F(w^t_*),w_t-w_* \rangle.
\end{align*}
Since $\nabla F(w^t_*)=0$ and $F(w^t_*)=F(w_*)=F_{min}$ for all $w_* \in \mathcal{W}^*$, we obtain 
% \begin{eqnarray*}
%  \mathbb{E}[Y_{t+1}| \mathcal{F}_t] &\leq  \mathbb{E}[Y_{t}| \mathcal{F}_t] -2\eta_t(1-\eta_tL)[F(w_t)-F(w_*)] + 2\eta^2_t \mathbb{E}[\|\nabla f(w^t_*;\xi_t) \|^2| \mathcal{F}_t]\\
%  &\leq \mathbb{E}[Y_{t}| \mathcal{F}_t] -2\eta_t(1-\eta_tL)[F(w_t)-F(w_*)] + 2\eta^2_t N, 
% \end{eqnarray*}
\begin{align*}
 \mathbb{E}[Y_{t+1}| \mathcal{F}_t] &\leq  \mathbb{E}[Y_{t}| \mathcal{F}_t] -2\eta_t(1-\eta_tL)[F(w_t)-F(w_*)] \\ &\qquad + 2\eta^2_t \mathbb{E}[\|\nabla f(w^t_*;\xi_t) \|^2| \mathcal{F}_t]\\
 &\leq \mathbb{E}[Y_{t}| \mathcal{F}_t] -2\eta_t(1-\eta_tL)[F(w_t)-F(w_*)] \\ &\qquad + 2\eta^2_t N, 
\end{align*}
where the last inequality follows from the definition of $N$. The lemma follows after substituting $E_t=F(w_t)-F(w_*)$.
\end{proof}

\section{W.p.1. Result} \label{sec:wp1}

\begin{lem}[\cite{BertsekasSurvey}]\label{prop_supermartingale}
Let $Y_k$, $Z_k$, and $W_k$, $k = 0,1,\dots$, be three sequences of random variables and let $\{\mathcal{F}_k\}_{k\geq 0}$ be a filtration, that is, $\sigma$-algebras such that $\mathcal{F}_k \subset \mathcal{F}_{k+1}$ for all $k$. Suppose that: 
\begin{itemize}
\item The random variables $Y_k$, $Z_k$, and $W_k$ are nonnegative, and $\mathcal{F}_k$-measurable. 
\item For each $k$, we have $
\mathbb{E}[Y_{k+1} | \mathcal{F}_k] \leq Y_k - Z_k + W_k$. 
\item There holds, w.p.1, 
\begin{gather*}
\sum_{k=0}^{\infty} W_k < \infty. 
\end{gather*}
\end{itemize}
Then, we have, w.p.1,
\begin{gather*}
\sum_{k=0}^{\infty} Z_k < \infty \ \text{and} \ Y_k \to Y \geq 0. 
\end{gather*}
\end{lem}

\textbf{Theorem \ref{thm_general_convex_wp1_01}}. 
%Suppose that Assumptions \ref{ass_smooth} and \ref{ass_convex} hold. 
\textit{Consider SGD with a stepsize sequence such that
\begin{align*}
0 < \eta_t \leq \frac{1}{L} \ , \ \sum_{t=0}^{\infty} \eta_t = \infty \ \text{and} \ \sum_{t=0}^{\infty} \eta_t^2 < \infty. 
\end{align*}
Then, the following holds w.p.1 (almost surely)
\begin{align*}
F(w_t) - F(w_*) \to 0, 
\end{align*}
where $w_{*}$ is any optimal solution of $F(w)$. }

\begin{proof}
%Let $\mathcal{F}_t = \sigma(w_0,w_1,\dots,w_t)$ be the $\sigma$-algebra generated by $w_0$, $w_1$, $\dots$, $w_t$. Since $w_{t+1} = w_{t} - \eta_t \nabla f(w_t;\xi_t)$, we have 
%\begin{align*}
%\mathbb{E} [ \| w_{t+1} - w_{*} \|^2 | \mathcal{F}_t ] &= \| w_{t} - w_{*} \|^2 - 2\eta_t \nabla F(w_t)^\top (w_t - w_*) + \eta_t^2 \mathbb{E} [ \| \nabla f(w_t;\xi_t) \|^2 | \mathcal{F}_t ] \\
%==================
%& \leq \| w_{t} - w_{*} \|^2 - 2\eta_t \nabla F(w_t)^\top (w_t - w_*) + 2 \eta_t^2 \mathbb{E} [ \| \nabla f(w_t;\xi_t) - \nabla f(w_*;\xi_t) \|^2 | \mathcal{F}_t ] \\ & \qquad \qquad + 2 \eta_t^2 \mathbb{E} [ \| \nabla f(w_*;\xi) \|^2 ] \\
%==================
%& \overset{\eqref{ineq_convex}}{\leq} \| w_{t} - w_{*} \|^2 - 2\eta_t \nabla F(w_t)^\top (w_t - w_*) + 2\eta_t^2 L \nabla F(w_t)^\top (w_t - w_*) \\ & \qquad \qquad + 2 \eta_t^2 \mathbb{E} [ \| \nabla f(w_*;\xi) \|^2 ] \\
%==================
%&= \| w_{t} - w_{*} \|^2 - 2\eta_t(1 - \eta_t L) \nabla F(w_t)^\top (w_t - w_*) + 2 \eta_t^2 \mathbb{E} [ \| \nabla f(w_*;\xi) \|^2 ] \\
%==================
%& \overset{\eta_t \leq \frac{1}{L},\eqref{eq:convex}}{\leq} \| w_{t} - w_{*} \|^2 - 2\eta_t(1 - \eta_t L) [ F(w_t) - F(w_*) ] + 2 \eta_t^2 \mathbb{E} [ \| \nabla f(w_*;\xi) \|^2 ]. 
%\end{align*}
%
%Therefore, 
%\begin{align*}
%\mathbb{E} [ \| w_{t+1} - w_{*} \|^2 | \mathcal{F}_t ] \leq \| w_{t} - w_{*} \|^2 - 2\eta_t(1 - \eta_t L) [ F(w_t) - F(w_*) ] + 2 \eta_t^2 \mathbb{E} [ \| \nabla f(w_*;\xi) \|^2 ]. \tagthis \label{eq_convex_001}
%\end{align*}
The following proof follows the proof in \cite{Nguyen2018}. From Lemma~\ref{lem:E_Yt} we recall the recursion 
$$ \mathbb{E}[Y_{t+1}| \mathcal{F}_t] \leq \mathbb{E}[Y_{t}| \mathcal{F}_t] -2\eta_t(1-\eta_tL)E_t + 2\eta^2_t N. $$
Let $Z_t= 2\eta_t(1-\eta_tL)E_t$ and $W_t= 2\eta^2_t N$.
Since $\sum_{t=0}^{\infty} W_t = \sum_{t=0}^{\infty} 2 \eta_t^2 N < \infty$, by Lemma \ref{prop_supermartingale}, we have w.p.1
\begin{gather*}
%\| w_{t} - w_{*} \|^2 \to W \geq 0, \\
\mathbb{E}[Y_{t}| \mathcal{F}_t] \to Y \geq 0, \\
\sum_{t=0}^{\infty} Z_t = \sum_{t=0}^{\infty} 2\eta_t(1 - \eta_t L) [ F(w_t) - F(w_*) ] < \infty. 
\end{gather*}

We want to show that $[ F(w_t) - F(w_*) ] \to 0$ w.p.1. Proving by contradiction, we assume that there exists $\epsilon > 0$ and $t_0$, s.t., $[ F(w_t) - F(w_*) ] \geq \epsilon$ for $\forall t \geq t_0$. Hence, 
\begin{align*}
\sum_{t=t_0}^{\infty} 2\eta_t(1 - \eta_t L) [ F(w_t) - F(w_*) ] & \geq \sum_{t=t_0}^{\infty} 2\eta_t(1 - \eta_t L) \epsilon \\
&= 2 \epsilon \sum_{t=t_0}^{\infty} \eta_t - 2 L \epsilon \sum_{t=t_0}^{\infty} \eta_t^2 = \infty. 
\end{align*}
This is a contradiction. Therefore, $[ F(w_t) - F(w_*) ] \to 0$ w.p.1. 

\end{proof}

\section{Convexity and Curvature}

\subsection{Function $\delta(\cdot)$} \label{sec:delta}

%TO DO Include max(0, definition delta) ... avoids minus infinity ..

For smooth functions $a: w\in \mathbb{R}^d \to [0,\infty)$ and $b: w\in \mathbb{R}^d \to [0,\infty)$ we define
$$\delta(\epsilon)=\underset{p:\mathbb{E}_p[a(w)]\leq \epsilon}{\mbox{sup}} \mathbb{E}_p[b(w)],$$
where $p$ represents a probability distribution over $w \in {\cal B}\subseteq \mathbb{R}^d$.
%where ${\cal B}$ is a convex set.
We assume $\delta(\epsilon)<\infty$ for $\epsilon\geq 0$.

The next lemmas show that function $\delta(\cdot)$ has a number of interesting properties that are useful for us when applied to 
$$a(w)=F(w)-F(w_*)=F(w)-F_{min}$$
and
$$b(w) = \underset{w_* \in \mathcal{W}^*}{\mbox{inf}} \|w-w_* \|^2.$$

\begin{lem} \label{lem:EY} Let $p$ be a probability distribution over $w \in {\cal B} \subseteq \mathbb{R}^d$. Then, $\mathbb{E}_p[b(w)]\leq \delta(\mathbb{E}_p[a(w)])$.
\end{lem}
\begin{proof}
Let $\epsilon=\mathbb{E}_p[a(w)]$. By the definition of function $\delta(\cdot)$, since $\mathbb{E}_p[a(w)]\leq \epsilon$, 
$$\mathbb{E}_p[b(w)]\leq \underset{\bar{p}:\mathbb{E}_{\bar{p}}[a(w)]\leq \epsilon}{\mbox{sup}} \mathbb{E}_{\bar{p}}[b(w)]=\delta(\epsilon)
=\delta(\mathbb{E}_p[a(w)]).$$
\end{proof}

To gain insight about distribution $p$ %over $w \in \mathbb{R}^d$
in the definition of function $\delta(\epsilon)$ we define
$$\mathcal{W}(\hat{\epsilon})=\{w \ : \ a(w)=\hat{\epsilon} \},$$
$$\hat{p}(w)=\mathrm{Pr}_p(w | w \in \mathcal{W}(\hat{\epsilon})),$$ and $$q(\hat{\epsilon})= \mathrm{Pr}_p(w \in \mathcal{W}(\hat{\epsilon})) = p(\mathcal{W}(\hat{\epsilon})).$$ 
Notice that from $q$ and $\hat{p}$ we can infer $p$ and vice versa.

By Bayes' rule, for $w\in \mathcal{W}(\hat{\epsilon})$,
$$
p(w) = q(\hat{\epsilon}) \hat{p}(w).
$$
Note that the space $\mathbb{R}^d$
%$$\mathcal{W}= \{w: p(w)>0 \}$$ 
is partitioned into the infinitely many subsets $\mathcal{W}(\hat{\epsilon})$ for $\hat{\epsilon}\geq 0$. This proves
\begin{eqnarray*}
\mathbb{E}_p[a(w)]&=&
\underset{\hat{\epsilon}\geq 0}{\int} \ \ \underset{w\in \mathcal{W}(\hat{\epsilon})}{\int} a(w) \hat{p}(w) q(\hat{\epsilon}) \ \mathrm{d}w\mathrm{d}\hat{\epsilon}
\end{eqnarray*}
(these integrals make sense because we assume smooth functions $a(w)$ and $b(w)$, which also implies that $\mathcal{W}(\hat{\epsilon})$ is connected and even convex if $a(\cdot)$ is convex).
By the definition of $\mathcal{W}(\hat{\epsilon})$, $a(w)=\hat{\epsilon}$ for
$w\in \mathcal{W}(\hat{\epsilon})$. This leads to the simplification
\begin{eqnarray*}
\mathbb{E}_p[a(w)]&=&
\underset{\hat{\epsilon}\geq 0}{\int}
 \hat{\epsilon}q(\hat{\epsilon}) \left\{
\underset{w\in \mathcal{W}(\hat{\epsilon})}{\int} \hat{p}(w) \ \mathrm{d}w\right\} \ \mathrm{d}\hat{\epsilon}
=
 \underset{\hat{\epsilon}\geq 0}{\int}  \hat{\epsilon}q(\hat{\epsilon})\mathrm{d}\hat{\epsilon}= \mathbb{E}_q[\hat{\epsilon}],
\end{eqnarray*}
where $q$ is considered a distribution over parameter $\hat{\epsilon}\in [0,\infty)$.

The above shows that we may rewrite
\begin{eqnarray*}
\delta(\epsilon)&=&
\underset{p:\mathbb{E}_p[a(w)]\leq \epsilon}{\mbox{sup}} \mathbb{E}_p[b(w)] \\
&=& \underset{q: \mathbb{E}_q[\hat{\epsilon}]\leq \epsilon}{\mbox{sup}} \ \
\underset{\hat{p} \text{ over } w\in \mathcal{W}(\hat{\epsilon})\cap {\cal B}}{\mbox{sup}}
\mathbb{E}_{q,\hat{p}}[b(w)] \\
&=& \underset{q: \mathbb{E}_q[\hat{\epsilon}]\leq \epsilon}{\mbox{sup}} \ \
\underset{\hat{p} \text{ over } w\in \mathcal{W}(\hat{\epsilon})\cap {\cal B}}{\mbox{sup}}
\mathbb{E}_{q}\left[
\underset{w\in \mathcal{W}(\hat{\epsilon})}{\int} \hat{p}(w) b(w) \ \mathrm{d}w \right].
\end{eqnarray*}

For $\mathcal{W}(\hat{\epsilon})\cap {\cal B}\neq \emptyset$, we implicitly define $w(\hat{\epsilon})$ as a solution of
\begin{equation} b(w(\hat{\epsilon})) = 
\underset{w\in \mathcal{W}(\hat{\epsilon})\cap {\cal B}}{\mbox{sup}} \ \ 
b(w) \label{segment}
\end{equation}
(here, we assume for simplicity that the supremum can be realized for some $w(\hat{\epsilon})$ and we note that the argument below can be made general with small adaptations).
For $\mathcal{W}(\hat{\epsilon})\cap {\cal B}\neq \emptyset$,  integral 
$$ \underset{w\in \mathcal{W}(\hat{\epsilon})}{\int} \hat{p}(w) b(w) \ \mathrm{d}w $$
is maximized by distribution $\hat{p}$ over ${\cal B}$ defined by 
$$ \hat{p}(w(\hat{\epsilon})) =1$$
and zero elsewhere. This proves the next lemma (in the supremum $q(\hat{\epsilon})$ should be chosen equal to $0$ if $\mathcal{W}(\hat{\epsilon})\cap {\cal B}= \emptyset$).

\begin{lem} \label{lem:altdelta} \begin{eqnarray*} \delta(\epsilon) &=&
\underset{q: \mathbb{E}_q[\hat{\epsilon}]\leq \epsilon}{\mathrm{sup}} \ \
\mathbb{E}_{q} [
b(w(\hat{\epsilon}))] 
=
\underset{q: \mathbb{E}_q[\hat{\epsilon}]\leq \epsilon}{\mathrm{sup}} \ \
\mathbb{E}_{q}\left[
\underset{w\in \mathcal{W}(\hat{\epsilon})\cap {\cal B}}{\mathrm{sup}} \ \
b(w) \right],
\end{eqnarray*}
where $q$ is a distribution over $\hat{\epsilon}\in [0,\infty)$ and 
$\mathcal{W}(\hat{\epsilon})=\{w \ : \ a(w)=\hat{\epsilon} \}$.
\end{lem}

 It turns out that 
 $$\rho(\hat{\epsilon})=b(w(\hat{\epsilon}))= \underset{w\in \mathcal{W}(\hat{\epsilon})\cap {\cal B}}{\mathrm{sup}} \ \
b(w)$$ 
 is increasing for $\mathcal{W}(\hat{\epsilon})\cap {\cal B}\neq \emptyset$, but it may have convex and concave parts. For this reason we are not able to simplify $\delta(\cdot)$ any further.
 
%We notice that if $a(w)$ is  convex, then  $\mathcal{W}(\hat{\epsilon})$ are convex sets for $\hat{\epsilon}\geq 0$. Since $a(w)$ is smooth, $w(\hat{\epsilon})$ is smooth (with properly chosen curve segments in the different sets $\mathcal{W}(\hat{\epsilon})$ in case (\ref{segment}) does not have a unique solution $w(\hat{\epsilon})$). 
%\textcolor{red}{Marten: Adapt red curve and use notation $a(w)$}

%\begin{figure}[h]
% \centering
% \includegraphics[width=0.45\textwidth]{./Figs/tikz/01.pdf}
% \includegraphics[width=0.45\textwidth]{./Figs/tikz/05.pdf}
%   \caption{Description of $w(\hat{\epsilon})$ function}
%  \label{fig:01}
% \end{figure}

Given the reformulation of $\delta(\epsilon)$ we are able to prove the following property of function $\delta(\cdot)$.

\begin{lem}
\label{lem:delta_function} Suppose that $a(w)=0$ implies both $b(w)=0$ and $w\in {\cal B}$, and suppose that there exists a $w_*$ such that $a(w_*)=0$. Then,
$\delta(0)=0$, $\delta(\epsilon)$ is increasing, and $\delta(\epsilon)$ is $\cap$-convex.
\end{lem}
\begin{proof} From Lemma \ref{lem:altdelta} we obtain the expression
$$\delta(0)= 
\underset{q: \mathbb{E}_q[\hat{\epsilon}]\leq 0}{\mathrm{sup}} \ \
\mathbb{E}_{q}\left[
\underset{w\in \mathcal{W}(\hat{\epsilon})\cap {\cal B}}{\mathrm{sup}} \ \
b(w) \right],$$ 
where $q$ is a distribution over $\hat{\epsilon}\in [0,\infty)$. Therefore, $\mathbb{E}_q[\hat{\epsilon}]\leq 0$ if and only if $\mathbb{E}_q[\hat{\epsilon}]=0$, i.e., the probability $\hat{\epsilon}=0$ is equal to 1 according to distribution $q$. This proves
$$\delta(0)= \underset{w\in \mathcal{W}(\hat{0})\cap {\cal B}}{\mathrm{sup}} \ \
b(w).
$$
By definition, if $w\in \mathcal{W}(\hat{0})$, then $a(w)=0$ and by our assumption $b(w)=0$ and $w\in {\cal B}$. Given the existence of $a(w_*)=0$, the set $\mathcal{W}(\hat{0})\cap {\cal B}$ is not empty, hence, $\delta(0)=0$.

%TO DO: May need here that ${\cal B}$ is connected. It may be $-\infty$ ... Then $q$ is also over $-\infty \cup [0,\infty)$. Also, $w(\hat{\epsilon})$ may become disconnected (not increasing only over the 'existing' parts). IMPORTANT 

We show that $\delta(.)$ is $\cap$-convex:
Lemma \ref{lem:altdelta} shows that for any $\epsilon_1, \epsilon_2>0$, there exists
distributions $q_1$ and $q_2$ such that $\mathbb{E}_{q_1}[\hat{\epsilon}] \leq \epsilon_1$, $\mathbb{E}_{q_2}[\hat{\epsilon}] \leq \epsilon_2$, and 
\begin{eqnarray*}
\delta(\epsilon_1) &=& \mathbb{E}_{q_1}[b(w(\hat{\epsilon}))] \text{ and }\\
\delta(\epsilon_2) &=& \mathbb{E}_{q_2}[b(w(\hat{\epsilon}))]. 
\end{eqnarray*}
(again we assume for simplicity that the supremums in the reformulation of function values $\delta(\epsilon_1)$ and $\delta(\epsilon_2)$ in Lemma \ref{lem:altdelta} can be realized for some $q_1$ and $q_2$; we note that the argument below can be made general with small adaptations).

Note that for $q=\alpha q_1 + (1-\alpha) q_2$, we have 
$$
\mathbb{E}_{q}[\hat{\epsilon}]=\alpha\mathbb{E}_{q_1}[\hat{\epsilon}]+ (1-\alpha) \mathbb{E}_{q_2}[\hat{\epsilon}] \leq \alpha \epsilon_1 + (1-\alpha)\epsilon_2 
$$
and
\begin{eqnarray*}
\mathbb{E}_{q}[b(w(\hat{\epsilon}))]&=& \alpha\mathbb{E}_{q_1}[b(w(\hat{\epsilon}))]+(1-\alpha)\mathbb{E}_{q_2}[b(w(\hat{\epsilon}))]\\
&=& \alpha \delta(\epsilon_1) + (1-\alpha) \delta(\epsilon_2).
\end{eqnarray*}

This shows that
\begin{eqnarray*}
\alpha \delta(\epsilon_1) + (1-\alpha) \delta(\epsilon_2) &=& \mathbb{E}_{q}[b(w(\hat{\epsilon}))] \\ &\leq& \underset{q:\mathbb{E}_{\bar{q}}[\hat{\epsilon}]\leq \alpha \epsilon_1 + (1-\alpha) \epsilon_2}{\mbox{sup}} \mathbb{E}_{\bar{q}}[b(w(\hat{\epsilon}))]\\
&=&\delta(\alpha \epsilon_1 + (1-\alpha) \epsilon_2). 
\end{eqnarray*}
\end{proof}

Let $\omega(x)$ be increasing and $\cap$-convex such that, for all $w \in \mathbb{R}^d$,
$$ \omega(a(w))\geq b(w).$$
Then, since $\omega(x)$ is $\cap$-convex, $\mathbb{E}[\omega(X)]\leq \omega(\mathbb{E}[X])$ for any real valued random variable. Hence,
$$\mathbb{E}_p[b(w)]\leq \mathbb{E}_p[\omega(a(w)]\leq \omega(\mathbb{E}_p[a(w)]).$$
If we assume $\mathbb{E}_p[a(w)]\leq \epsilon$, then
$$\mathbb{E}_p[b(w)]\leq \omega(\epsilon)$$
because $\omega(\cdot)$ is increasing.
By the definition of $\delta(\cdot)$,
$$\delta(\epsilon)= \underset{p:\mathbb{E}_p[a(w)]\leq \epsilon}{\mbox{sup}} \mathbb{E}_p[b(w)]\leq \omega(\epsilon).$$

The other way around holds true as well: Let $\omega(x)$ be increasing and $\cap$-convex such that $\omega(\epsilon)\geq \delta(\epsilon)$ for $\epsilon\geq 0$. Then, by Lemma \ref{lem:EY} using point distribution $p(w)=1$ and $0$ elsewhere, $b(w)\leq \delta(a(w))$ for $w\in {\cal B}$. Since $\omega(\epsilon)\geq \delta(\epsilon)$, we have $b(w)\leq \omega(a(w))$.
We have the following lemma.

\begin{lem} \label{lem:omega} Let $\omega(x)$ be increasing and $\cap$-convex. Then,
$\omega(a(w))\geq b(w)$ for all $w \in {\cal B}\subseteq \mathbb{R}^d$ if and only if $\delta(\epsilon)\leq \omega(\epsilon)$ for all $\epsilon\geq 0$.
\end{lem}

The above lemma shows that $\delta(\cdot)$ is the 'minimal' increasing and $\cap$-convex function with the property $\omega(a(w))\geq b(w)$ for all $w$.

\subsection{Relating Expectations of $a(\cdot)$ and $b(\cdot)$} \label{sec:expab}

%TO DO Each subsection: connect which lemmas will prove what statements in the main body.

%Figures~\ref{fig:02} and \ref{fig:04} describe the $\cap$-convexity of $\omega(\epsilon)$. Figure \ref{fig:04} shows that, for all $x>0$ and $y>0$,
We start by noting that, for all $x>0$ and $y>0$,
\begin{equation} \omega(y)\leq \omega(x) + \omega'(x) (y-x)\label{eqprop} \end{equation}
(irrespective of whether $y\leq x$ or $y\geq x$).

%\textcolor{red}{Marten:Adapt Figure~\ref{fig:02} with $a(w)$ -- better is to merge the two figures (leave this one out).}

%\begin{figure}[!h]
%\centering
%\includegraphics[scale=0.7]{./Figs/tikz/02.pdf}
%\caption {Function $\delta(\cdot)$.}
% \label{fig:02}
%\end{figure}

%\begin{figure}[h]
% \centering
% \includegraphics[width=0.32\textwidth]{./Figs/tikz/03.pdf}
% \includegraphics[width=0.32\textwidth]{./Figs/tikz/04.pdf}
%   \caption{$\cap$-convexity of $\delta(\epsilon)$.}
%  \label{fig:04}
% \end{figure}
 
 \begin{lem} \label{lem:delta} Let $\omega(\cdot)$ be increasing and $\cap$-convex with $\omega(0)=\tau\geq 0$. %Then, for all $x>0$ and $y>0$,
%$\omega(y)\leq \omega(x) + \omega'(x) (y-x)$. Furthermore,
Then,
\begin{eqnarray}
\omega'(x) &\leq& \frac{\omega(x)-\tau}{x} \mbox{ for all } x>0, \mbox{ and } \label{eqprop1} \\
c_{\alpha}(e) \frac{\omega(x)}{x} &\leq& \omega'(x) \mbox{ for all } 0< \alpha\leq x\leq e \mbox{ and } 0<e< \mbox{sup}\{ z\geq 0 \ : \ \omega'(z)\neq 0\},  \label{eqprop2} \\
\mbox{ where } && \nonumber \\
 c_{\alpha}(e) &=& \underset{x\in [\alpha,e]}{\mathrm{inf}} (\frac{\omega(2x)}{\omega(x)} -1) > 0.
 \end{eqnarray}
 We notice that (1) $\mathrm{sup}\{ z\geq 0 \ : \ \omega'(z)\neq 0\}>0$ if and only if $\omega(x)$ is not the all zero function, and (2) by combining (\ref{eqprop1}) and (\ref{eqprop2}), $c(e)\leq 1$.
\end{lem}
\begin{proof}
We start by noting that, for all $x>0$ and $y>0$,
\begin{equation} \omega(y)\leq \omega(x) + \omega'(x) (y-x)\label{eqprop} \end{equation}
(irrespective of whether $y\leq x$ or $y\geq x$).

By substituting $y=0$ in (\ref{eqprop}), $\tau=\omega(0) \leq \omega(x) + \omega'(x)[0-x]$. Therefore, $x\cdot \omega'(x)\leq \omega(x)-\tau$ and (\ref{eqprop1}) follows. 

By substituting $y=2x$ in (\ref{eqprop}), $\omega(2x) \leq \omega(x) + \omega'(x)[2x-x]$. Therefore, $\omega(2x)-\omega(x)\leq x\cdot \omega'(x)$. By the definition of $c_\alpha(e)$, 
$c_\alpha(e) \leq \frac{\omega(2x)}{\omega(x)} -1$ or equivalently $(c_\alpha(e)+1)\omega(x) \leq \omega(2x)$. Combining inequalities yields $c_\alpha(e)\omega(x)\leq x\cdot \omega'(x)$ which proves (\ref{eqprop2}).

Since $\omega(x)$ is increasing, $\omega(2x)\geq \omega(x)$ and $c_\alpha(e)\geq 0$. If $c_\alpha(e)=0$, then there exists an $x\in [\alpha,e]$ such that $\frac{\omega(2x)}{\omega(x)} -1=0$, i.e., $\omega(2x)=\omega(x)$. This implies that $\omega(\cdot)$ is constant on the non-empty interval $[x,2x]$. Together with $\omega(\cdot)$ being increasing and $\cap$-convex this implies that $\omega(\cdot)$ is constant on  $[x,\infty)$, hence, $\omega'(z)=0$ for $z\geq x$. This means that 
$$ e< \mbox{sup}\{ z\geq 0 \ : \ \omega'(z)\neq 0\} \leq x,$$
contradicting $x\in [\alpha,e]$. So, $c_\alpha(e)\neq 0$.
\end{proof}

By substituting $y=\mathbb{E}_p[a(w)]$ in (\ref{eqprop}), we obtain, for all $x\geq 0$,
\begin{equation}
\label{eq:xxxx1111111}
\omega(\mathbb{E}_p[a(w)])\leq \omega(x) + \omega'(x)[\mathbb{E}_p[a(w)]-x].
\end{equation}
Lemma \ref{lem:EY} and Lemma \ref{lem:omega}, where we assume $\omega(a(w))\geq b(w)$ for all $w$,  prove 
\begin{equation} \mathbb{E}_p[b(w)] \leq \delta(\mathbb{E}_p[a(w)])\leq \omega(\mathbb{E}_p[a(w)]) \label{eqA} \end{equation}
(a more direct proof is given below and is also in the main body).
Combination of (\ref{eq:xxxx1111111}) and (\ref{eqA}) yields
\begin{align*}
\mathbb{E}_p[b(w)] &\leq \omega(\mathbb{E}_p[a(w)])
%\overset{Eq~(\ref{eq:xxxx1111111})}{\leq} 
\leq \omega(x) + \omega'(x)[\mathbb{E}_p[a(w)]-x]\\
&= (\omega(x)-\omega'(x)x) + \omega'(x)\mathbb{E}_p[a(w)].
\end{align*}

We infer from Lemma \ref{lem:delta} that
$$ \omega'(x) \leq \frac{\omega(x)-\tau}{x}.$$
Lemma \ref{lem:delta} also shows that for $0< \alpha \leq x\leq e$ and $e$ small enough
$$ 
\frac{\omega(x)}{\omega'(x)} -x \leq (c_\alpha(e)^{-1}-1)x,
$$
 where $c_\alpha(e)> 0$.
Applying these results to the above derivation yields, for $\alpha\leq x\leq e$,
\begin{eqnarray*}
\frac{x\mathbb{E}_p[b(w)]}{\omega(x)-\tau}
&\leq& \frac{\mathbb{E}_p[b(w)]}{\omega'(x)}
\leq (\frac{\omega(x)}{\omega'(x)}-x) + \mathbb{E}_p[a(w)]
\leq (c_\alpha(e)^{-1}-1)x + \mathbb{E}_p[a(w)].
\end{eqnarray*}
If we assume $\omega'(\epsilon)\neq 0$ for all $\epsilon\geq 0$, then 
$$ \mbox{sup}\{ z\geq 0 \ : \ \omega'(z)\neq 0\} =\infty,$$
hence, $e$ is unrestricted. By using
$$ 
\underset{e\in [\alpha,\infty)}{\mathrm{sup}} \ \
\underset{x\in [\alpha,e]}{\mathrm{inf}} \frac{\omega(2x)}{\omega(x)} -1 $$
instead of $c_\alpha(e)$ we obtain the best bound:

\begin{lem} \label{lem:expineq} Let $\omega: [0,\infty) \to [0,\infty)$ be $\cap$-convex (i.e. $\omega''(\epsilon)<0$) and strictly increasing (i.e., $\omega'(\epsilon)>0$) with $\omega(0)=\tau\geq 0$. If $\omega(a(w))\geq b(w)$ for all $w \in \mathbb{R}^d$, then, (1) for all $0< \alpha\leq x$,
$$\frac{x\mathbb{E}_p[b(w)]}{\omega(x)-\tau} \leq \frac{2-c_\alpha}{c_\alpha-1}x + \mathbb{E}_p[a(w)],$$
% 1/(c-1) -1 = (2-c)/(1-c)
where
$$
 c_\alpha= 
 \underset{e\in [\alpha,\infty)}{\mathrm{sup}} \ \
 \underset{x\in [\alpha,e]}{\mathrm{inf}} \frac{\omega(2x)}{\omega(x)} >1 ,
 $$
  and (2) for all $0<x$,
  $$ \frac{\mathbb{E}_p[b(w)]}{\omega'(x)}
\leq (\frac{\omega(x)}{\omega'(x)}-x) + \mathbb{E}_p[a(w)].$$
\end{lem}

\begin{proof} We repeat a more condensed proof for (2):
Since $\omega(a(w))\geq b(w)$ for all $w$, 
$$\mathbb{E}_p[b(w)]\leq \mathbb{E}_p[\omega(a(w))].$$
Since $\omega(\cdot)$ is $\cap$-convex,
$$\mathbb{E}_p[\omega(a(w))]\leq \omega(\mathbb{E}_p[a(w)]).$$
By substituting $y=\mathbb{E}_p[a(w)]$ in (\ref{eqprop}), we obtain
$$\omega(\mathbb{E}_p[a(w)])\leq \omega(x) + \omega'(x)[\mathbb{E}_p[a(w)]-x].$$
Rearranging terms and dividing by $\omega'(x)$ proves property (2).
\end{proof}

%TO DO: Want $x$ large so that left hand side gets tight, turns out to be important, while constraining $x$ such that
%$$\frac{2-c_\alpha}{c_\alpha-1}x \leq \eta_t.$$
%Clearly (hmmm) $c_\alpha$ is decreasing in $\alpha$, hence $(c_\alpha^{-1}-1)$ is increasing. We know $\omega(x)/\omega'(x) -x$ is increasing, so, it makes sense to make $x$ as large as possible to fit constraint -- and see lemma 7 best is to take $\alpha=x$  ....???

We will apply the above lemma to the class of strictly increasing and $\cap$-convex functions
$$\omega_{h,r,\mu,\tau}(x)
= \left\{ 
\begin{array}{ll}
\tau + \frac{2}{\mu}(x/r)^h, & \mbox{ if } x\leq r, \mbox{ and} \\
\tau + \frac{2}{\mu} + \frac{2}{\mu}h ((x/r)-1), & \mbox{ if } x>r,
\end{array}\right.
$$
for $h\in (0,1]$, $r>0$, $\mu>0$, and $\tau\geq 0$. Notice that $\omega_{h,r, \mu,\tau}(0)=\tau$.

Function $\omega_{h,r, \mu,\tau}(x)$ is curved like $x^h$ for $x$ close enough to zero up to $x\leq r$. For $x>r$, values $\omega_{h,r, \mu,\tau}(x)$ are chosen as large as possible under the constraint that $\omega_{h,r, \mu,\tau}(\cdot)$ remains $\cap$-convex (i.e., $\omega_{h,r, \mu,\tau}(x)$ is equal to the tangent of $\frac{2}{\mu}(x/r)^h$ at $x=r$).

In our analysis of $c_\alpha$ we consider three cases. Let $x\geq \alpha$.
First, suppose that $2x\leq r$. Then,
$$
\frac{\omega_{h,r, \mu,\tau}(2x)}{\omega_{h,r, \mu,\tau}(x)} = 
\frac{\tau + \frac{2}{\mu}(2x/r)^h}{\tau + \frac{2}{\mu}(x/r)^h}
= 1 + \frac{2^h-1}{(\mu\tau/2)(r/x)^h + 1}.$$
This is minimized for $x$ as small as possible. Assuming $2\alpha\leq r$, allows $x=\alpha$ and achieves
\begin{equation}
\underset{x\in [\alpha,r/2]}{\mathrm{inf}}\frac{\omega_{h,r, \mu,\tau}(2x)}{\omega_{h,r, \mu,\tau}(x)} = 
1 + \frac{2^h-1}{(\mu\tau/2)(r/\alpha)^h + 1}. \label{eq:case1}
\end{equation}

Second, suppose that $x\leq r\leq 2x$. Then,
\begin{equation}
\frac{\omega_{h,r, \mu, \tau}(2x)}{\omega_{h,r, \mu, \tau}(x)}
=
\frac{\tau+\frac{2}{\mu} +\frac{2}{\mu}h ((2x/r)-1)}{\tau+\frac{2}{\mu}(x/r)^h}.
\label{eq:case2a}
\end{equation}
Setting variable $y=r/x \in [1,2]$ and taking the derivative with respect to $y$ (and grouping terms) gives
$$
\frac{h \frac{2}{\mu} y^{-h-1} \{\frac{2}{\mu}(1-h)[ 1-2y^{-1}] -\tau [2y^{h-1}-1]\} }{\tau+\frac{2}{\mu}y^{-h}}.$$
The term $[1-2y^{-1}]\leq 0$ for $y\in [1,2]$ and the term $[2y^{h-1}-1]\geq 0$ for $y\in [1,2]$. This shows that the derivative above is $\leq 0$ for $y\in[1,2]$. Therefore (\ref{eq:case2a}) is minimized for $y=r/x=2$; we again assume $2\alpha\leq r$.
This achieves
\begin{equation}\underset{x\in [r/2,r]}{\mathrm{inf}}
\frac{\omega_{h,r, \mu, \tau}(2x)}{\omega_{h,r, \mu, \tau}(x)}
=
\frac{\tau+\frac{2}{\mu}}{\tau+\frac{2}{\mu}2^{-h}}
=
1+ \frac{2^h-1}{\tau \frac{\mu}{2}2^h +1}.
\label{eq:case2}
\end{equation}

Third, suppose that $r\leq x$. Then,
\begin{align*}
\frac{\omega_{h,r, \mu, \tau}(2x)}{\omega_{h,r, \mu, \tau}(x)}
&= \frac{\tau+\frac{2}{\mu} +\frac{2}{\mu}h ((2x/r)-1)}{\tau+\frac{2}{\mu} +\frac{2}{\mu}h ((x/r)-1)}
=\frac{\tau\frac{\mu}{2}+1 + h ((2x/r)-1)}{\tau\frac{\mu}{2}+1 + h((x/r)-1)} \\
&= 1 + \frac{h}{(\tau\frac{\mu}{2}+1-h)(r/x)+h}.    
\end{align*}
This is minimized for $r/x$ as large as possible, i.e., $r=x$, which achieves
\begin{equation}\underset{x\in [r,\infty)}{\mathrm{inf}}
\frac{\omega_{h,r, \mu, \tau}(2x)}{\omega_{h,r, \mu, \tau}(x)}
=
1 + \frac{h}{(\tau\frac{\mu}{2}+1-h)+h}
=
1 + \frac{h}{\tau\frac{\mu}{2}+1}.
\label{eq:case3}
\end{equation}

The above analysis shows that the three cases (\ref{eq:case1}), (\ref{eq:case2}), and (\ref{eq:case3}) neatly fit together in that 
$$ \frac{\omega_{h,r, \mu, \tau}(2x)}{\omega_{h,r, \mu, \tau}(x)} \mbox{ is increasing in } x\geq \alpha.$$
This proves that $c_\alpha$ is equal to (\ref{eq:case1});
$$ c_\alpha = 1 + \frac{2^h-1}{(\mu\tau/2)(r/\alpha)^h + 1}.$$
In Lemma \ref{lem:expineq} we need
$$ 
\frac{2-c_\alpha}{c_\alpha-1} = 
\frac{(\mu\tau/2)(r/\alpha)^h + 1}{2^h-1} -1 = 
\frac{(\mu\tau/2)(r/\alpha)^h + 2-2^h}{2^h-1}.$$

Lemma \ref{lem:expineq} provides the following results: (1) For $0<\alpha\leq x\leq r$ with $\alpha\leq r/2$,
$$\frac{\mu}{2}r^h x^{1-h}\mathbb{E}_p[b(w)] \leq \frac{(\mu\tau/2)(r/\alpha)^h + 2-2^h}{2^h-1}x + \mathbb{E}_p[a(w)].$$
By substituting $\alpha=x$, the tightest inequality is obtained:
$$\frac{\mu}{2}r^h x^{1-h}\mathbb{E}_p[b(w)] \leq \frac{(\mu\tau/2)(r/x)^h + 2-2^h}{2^h-1}x + \mathbb{E}_p[a(w)].$$
  (2) By substituting exact expressions for $\omega'(x)$ and $\omega(x)$ in Lemma \ref{lem:expineq}, we obtain for all $0<x\leq r$,
  $$ \frac{1}{h} \frac{\mu}{2} r^h x^{1-h}\mathbb{E}_p[b(w)]
\leq \frac{(\mu\tau/2)(r/x)^h +1  - h }{h} x + \mathbb{E}_p[a(w)].$$
This shows that the asymptotic dependency on $x$  obtained by the more accurate derivation in (2) is the same as for the  slightly less tight derivation giving (1). The technique that led to (1) may be a useful tool in analyzing other functions $\omega(.)$.

We summarize (2) in the following lemma:

\begin{lem}
 Let
 $$\omega_{h,r,\mu,\tau}(x)
= \left\{ 
\begin{array}{ll}
\tau + \frac{2}{\mu}(x/r)^h, & \mbox{ if } x\leq r, \mbox{ and} \\
\tau + \frac{2}{\mu} + \frac{\mu}{2}h ((x/r)-1), & \mbox{ if } x>r.
\end{array}\right.
$$
Then,
 $$ \frac{1}{h} \frac{\mu}{2} r^h x^{1-h}\mathbb{E}_p[b(w)]
\leq \frac{(\mu\tau/2)(r/x)^h +1  - h }{h} x + \mathbb{E}_p[a(w)].$$
\end{lem}

In our analysis of the convergence rate we need
$$ v(\eta) = \mathrm{sup} \{ \frac{1}{\omega'(x)} \ : \ \frac{\omega(x)}{\omega'(x)}-x\leq \eta \}.$$
Notice that the
derivative of $\frac{\omega(x)}{\omega'(x)}-x$ is equal to 
$$ \frac{-\omega(x)\omega''(x)}{\omega'(x)^2}\geq 0, $$
and the derivative $\frac{1}{\omega'(x)}$ is equal to
$$ \frac{-\omega''(x)}{\omega'(x)^2}\geq 0.$$
This implies that $v(\eta)$ is increasing and is alternatively defined as
$$v(\eta) = \frac{1}{\omega'(x)} \mbox{ where } \eta = \frac{\omega(x)}{\omega'(x)}-x.$$
%This definition allows us to compute
%$$v'(\eta) = -\frac{\omega''(x)}{\omega'(x)^2} \frac{dx}{d\eta} $$
%with
%$$ d\eta = -\frac{ \omega(x)\omega''(x)}{\omega'(x)^2} dx.$$
%Combining both equations yields
%$$v'(\eta) = -\frac{\omega''(x)}{\omega'(x)^2} \frac{dx}{d\eta} 
%= \frac{1}{\omega(x)}.
%$$
%Plugging this back into the definition proves
%$$ \eta = \frac{\omega(x)}{\omega'(x)}-x = \frac{v(\eta)}{v'(\eta)} -x$$
%which solves for $x$. This gives another alternative definition:
%$$v(\eta) = 1/\omega'(\frac{v(\eta)}{v'(\eta)}-\eta).$$

%For small $\eta$ we have (we use $v(0)=0$)
%$$ \frac{v(\eta)}{v'(\eta)}-\eta \approx \frac{v(\eta)'\eta +v''(\eta)\eta^2/2}{v'(\eta)}-\eta
%=\frac{v''(\eta)\eta^2}{2v'(\eta)}.$$
%We can lower bound $v(\eta)$ as follows:
%$$v(\eta) = 1/\omega'(\frac{v(\eta)}{v'(\eta)}-\eta) 
%\geq \frac{ \frac{v(\eta)}{v'(\eta)}-\eta}{\omega(\frac{v(\eta)}{v'(\eta)}-\eta)}

For $\omega_{h,r,\mu,\tau}$ we have
$$v(\eta) = \frac{1}{h} \frac{\mu}{2} r^h x^{1-h} \mbox{ where } \eta = \frac{(\mu\tau/2)(r/x)^h +1  - h }{h} x.$$

If $\tau\neq 0$, then
$$x = \frac{\eta h}{(\mu\tau/2)(r/x)^h +1  - h}\leq \frac{\eta h}{(\mu\tau/2)(r/x)^h},$$
hence,
$$x^{1-h}\leq \frac{2 \eta h }{\mu\tau r^h}.$$
%This 
%yields
%$$x \leq \eta^{1/(1-h)} \left(\frac{2 \eta h }{\mu\tau r^h}\right)^{1/(1-h)}$$
%and 
We get the upper bound
$$v(\eta) = \frac{1}{h} \frac{\mu}{2} r^h x^{1-h} \leq
\frac{1}{h} \frac{\mu}{2} r^h \frac{2 \eta h }{\mu\tau r^h} = 
   \frac{\eta  }{\tau }.$$
   This upper bound is tight for  small $\eta$.
   
   If $\tau=0$, then
 $$x = \frac{\eta h}{1  - h},$$
hence,  
$$v(\eta) = \frac{1}{h} \frac{\mu}{2} r^h x^{1-h} =
\frac{1}{h} \frac{\mu}{2} r^h \left(\frac{\eta h}{1  - h}\right)^{1-h}
= \frac{\mu}{2}  h^{-h}(1-h)^{-(1-h)}  r^h \eta^{1-h}.$$

\begin{lem} \label{lem17}
For $\omega_{h,r,\mu,\tau}$, 
$$v(\eta)\leq \frac{\eta  }{\tau } \mbox{ if } \tau\neq 0 $$
and
$$v(\eta) =\frac{\mu}{2}  h^{-h}(1-h)^{-(1-h)}  r^h \eta^{1-h}  \mbox{ if } \tau= 0.$$
\end{lem}

Notice that taking the limit $h\downarrow 0$ for $\tau=0$ gives 
$$v(\eta) = \frac{\mu}{2} \eta.$$
The limit $h=1$ gives 
$$v(\eta)= \frac{\mu}{2}r,$$
where $r=1$ corresponds to $\mu$-strongly objective functions.

In our definition and analysis of curvature (in the main text) we use the functions
$$\omega_{h,r,\mu}(x)= \omega_{h,r,\mu h, \tau=0}(x)
= \left\{ 
\begin{array}{ll}
\frac{2}{\mu h}(x/r)^h, & \mbox{ if } x\leq r, \mbox{ and} \\
\frac{2}{\mu h} + \frac{2}{\mu} ((x/r)-1), & \mbox{ if } x>r,
\end{array}\right.
$$
for $h\in (0,1]$, $r>0$, and $\mu>0$. We conclude from Lemma \ref{lem17} that for these functions,
$$v(\eta) =\frac{\mu h}{2}  h^{-h}(1-h)^{-(1-h)}  r^h \eta^{1-h}. $$
%and we extend this definition by defining functions
%$$\omega_{0,r,\mu}(x)= \frac{2}{\mu}(1+x/r).$$
%The latter functions correspond to any convex objective function (we know that a convex objective function is $\omega_{0,r,\mu}$-convex for some $r$ and $\mu$).
%Both sets of functions have
%$$v(\eta) =\frac{\mu}{2}  h^{-h}(1-h)^{-(1-h)}  r^h \eta^{1-h}.$$

%\textcolor{red}{interpretation tau =0 and not equal 0; in main paper we only discuss $\tau=0$ and $w_l$; talk about curvature $h$}
%===== STUCK HERE ==== Define $\omega$-convexity. Derive convergence rate for this.
%TO DO: correct derivations, distinguish $\tau=0$ and $\tau>0$. Each has its own recurrence. Show that $\tau>0$ essentially corresponds to "$h=0$". Show that a function with $\tau>0$ always exists. Translate $\tau=0$ to assumption wrt $h$ (include $h=0$ case). Show strong convexity and convexity. Explain that strong convexity is exactly $h=1$! Convexity with diminishing stepsizes needs assumption on $w$. Classical proof gives the same and does not assume anything. Solve recurrence and find optimal step size, $Y_t$ and $E_t$ (use integral formulation).

\subsection{Example Curvature $h=1/2$} \label{sec:example}

Let 
$$F(w)=H(w) + \lambda G(w)$$
be our objective function where $\lambda>0$,   $H(x)$ is a convex function, and
$$ G(w) = \sum_{i=1}^d [e^{w_i}+e^{-w_i}-2-\alpha w_i^2] \mbox{ with } \alpha=1.$$
Since $H(w)$ is convex, 
$$H(w)-H(w')\geq \langle \nabla H(w'),(w - w') \rangle.$$
If we can prove, for all $w, w' \in \mathbb{R}^d$,
\begin{equation}
G(w)-G(w')\geq \langle \nabla G(w'),(w - w') \rangle + \gamma \|w-w'\|^{2/h}, \label{curvG}
\end{equation}
then both inequalities can be added to obtain
$$F(w)-F(w')\geq \langle \nabla F(w'),(w - w') \rangle +\lambda \gamma \|w-w'\|^{2/h}.$$
For $w'=w_*\in \mathcal{W}^*$, i.e., $\nabla F(w_*)=0$, we obtain
$$F(w)-F(w_*) \geq \lambda \gamma \|w-w_*\|^{2/h}.$$
Since this holds for all $w_*$ and $F(w_*)=F_{min}$, we have
$$F(w)-F(w_*) \geq \lambda \gamma  \left\{\underset{w_* \in \mathcal{W}^*}{\mbox{inf}} \|w-w_* \|^2\right\}^{1/h}.$$
Hence $F$ is $\omega$-convex over $\mathbb{R}^d$  for $\omega(x)=\frac{2}{\mu h} x^h$ with $\mu= \frac{2\lambda \gamma}{ h} $. We conclude that $F$ has curvature $h$ over $\mathbb{R}^d$.

We will prove (\ref{curvG}) for $h=1/2$. 
We derive
$$G(w)-G(w') = \sum_{i=1}^d \left\{ [e^{w_i}+e^{-w_i}-2-\alpha w_i^2 ]-[e^{w'_i}+e^{-w'_i}-2-\alpha w_i'^2 ]\right\}$$
and
$$\langle \nabla G(w'),(w - w') \rangle
= \sum_{i=1}^d [e^{w'_i}-e^{-w'_i}-\alpha 2 w'_i]\cdot (w_i-w'_i).$$
Let $v_i=w_i-w'_i$ and substitute $w'_i=w_i-v_i$ in the above equations. Then,
\begin{align*}
& [G(w)-G(w')] - \langle \nabla G(w'),(w - w') \rangle
\\ & \qquad \qquad = 
\sum_{i=1}^d \left\{e^{w_i}[1-e^{-v_i}-e^{-v_i}v_i]+
e^{-w_i}[1-e^{v_i}+e^{v_i}v_i] -\alpha v_i^2\right\} \tagthis \label{Gw}  
\end{align*}
and we want to prove that this is at least
$$ \geq \gamma \left\{\sum_{i=1}^d v_i^2 \right\}^{1/h} =\gamma \|w-w'\|^{2/h}.$$
Differentiating (\ref{Gw}) with respect to $w_i$ yields
\begin{equation} e^{w_i}[1-e^{-v_i}-e^{-v_i}v_i] - e^{-w_i}[1-e^{v_i}+e^{v_i}v_i]. \label{diffw}
\end{equation}
Notice that $1-e^{-v_i}-e^{-v_i}v_i\geq 0$ since $e^{v_i}\geq 1+v_i$ for all $v_i$. Also notice that $1-e^{v_i}+e^{v_i}v_i\geq 0$ since $e^{-v_i}\geq 1-v_i$ for all $v_i$. This shows that (\ref{Gw}) is minimized for $w_i$ for which (\ref{diffw}) is equal to $0$, i.e.,
$$e^{w_i} = \sqrt{\frac{1-e^{v_i}+e^{v_i}v_i}{1-e^{-v_i}-e^{-v_i}v_i}}.$$
Plugging this back into (\ref{Gw}) shows that (\ref{Gw}) is at most
\begin{align*}
& \sum_{i=1}^d 2\sqrt{[1-e^{-v_i}-e^{-v_i}v_i][1-e^{v_i}+e^{v_i}v_i]}-\alpha  v_i^2
\\ & \qquad \qquad =\sum_{i=1}^d 2\sqrt{2-v_i^2-(e^{v_i}+e^{-v_i})+v_i(e^{v_i}-e^{-v_i})}-\alpha v_i.  
\end{align*}
We substitute the Taylor series expansion of $e^{v_i}$ and $e^{-v_i}$
and get
$$\sum_{i=1}^d 2\sqrt{2 \sum_{j=2}^\infty \frac{(2j-1)}{(2j)!} v_i^{2j}}-\alpha v_i^2.$$
The $i$-th term is at least $\bar{\alpha} v_i^4$ if
\begin{equation}2 \sum_{j=2}^\infty \frac{(2j-1)}{(2j)!} v_i^{2j} \geq (\bar{\alpha} v_i^4 +\alpha v_i^2)^2/4 = \frac{\alpha^2}{4} v_i^4 + \frac{\alpha \bar{\alpha}}{2}v_i^6 + \frac{\bar{\alpha}^2}{4}v_i^8.
\label{eqsum}
\end{equation}
The first three terms of the infinite sum are
$$ \frac{v_i^4}{4}+ \frac{v_i^6}{72} + \frac{v_i^8}{2880}.$$
So, if we set
$$\alpha = 1 \mbox{ and } \bar{\alpha}= \frac{1}{36},$$
then (\ref{eqsum}) is satisfied. So,
$$ [G(w)-G(w')] - \langle \nabla G(w'),(w - w') \rangle
\geq \sum_{i=1}^d \bar{\alpha}v_i^4.$$
Since the 2-norm and 4-norm satisfy 
$$\|w-w'\|_2=\left( \sum_{i=1}^d v_i^2 \right)^{1/2} \leq d^{1/4} \left( \sum_{i=1}^d v_i^4 \right)^{1/4},$$
we obtain
$$ \sum_{i=1}^d \bar{\alpha}v_i^4 \geq \frac{\bar{\alpha}}{d} \|w-w'\|^4.$$
This proves (\ref{curvG}) for $h=1/2$ and 
$$\gamma = \frac{\bar{\alpha}}{d} = \frac{1}{36d}.$$
Hence, $F$ is $\omega$-convex over $\mathbb{R}^d$  for $\omega(x)=\frac{2}{\mu h} x^h$ with $\mu= \frac{2}{\lambda \gamma h} $.

\begin{thm}
Let 
$$F(w)=H(w) + \lambda G(w)$$
be our objective function where $\lambda>0$,   $H(x)$ is a convex function, and
$$ G(w) = \sum_{i=1}^d [e^{w_i}+e^{-w_i}-2-  w_i^2].$$ 
Then, $F$ is $\omega$-convex over $\mathbb{R}^d$  for $\omega(x)=\frac{2}{\mu h} x^h$ with $h=1/2$ and
$\mu= \frac{2 \lambda \gamma}{h} = \frac{\lambda}{9d}$. 
\end{thm}

As the derivation shows, it is not possible to prove a curvature $>1/2$ (the bounds are tight in that we can always find an example which violates a larger curvature). 

The associated $v(\eta)$ as defined in (\ref{V}) is equal to
$$v(\eta) =\beta h \eta^{1-h} \mbox{ with } \beta = \frac{\mu}{2}  h^{-h}(1-h)^{-(1-h)},$$
for $\mu\geq 0$. In effect the $r^h$ term of $\omega_{h,\mu,r}$ is absorped in $\mu$ and $r\rightarrow \infty$. Function $\omega$ in the above theorem is equal to $\lim_{r\rightarrow \infty} \omega_{h,\mu/r^h,r}$.

\section{Proof Convergence Rate} \label{sec:convergence}

\begin{lem} Let $n(\cdot)$ be a decreasing step size function representing $n(t)=\eta_t$. Define
$$M(y)=\int_{x=0}^y n(x)v(n(x)) dx \mbox{ and }  C(t) = \exp(-M(t)) \int_{x=0}^{t} \exp(M(x))n(x)^{2}  dx .$$
Then recurrence 
$$
\mathbb{E}[Y_{t+1}] \leq (1- \eta_t v(\eta_t))\mathbb{E}[Y_{t}] +  (2N+1)\eta^2_t $$
implies
$$\mathbb{E}[Y_{t}] \leq A \cdot C(t) + B \cdot \exp(-M(t))
$$
for constants $A=(2N+1) \exp(n(0))$ and $B=(2N+1)\exp(M(1))n(0)^{2}  + \mathbb{E}[Y_{0}]$ (depending on parameter $N$ and starting vector $w_0$).
\end{lem}
%TO DO switch from 1 to 0
\begin{proof}
 We first define some notation:
$$ y_t = \mathbb{E}[Y_{t}], \ \ n(t) = \eta_t. %, \mbox{ and } \gamma(t) = (2N+c(e)^{-1}-1)\eta^2_t = \frac{4(2N+c(e)^{-1}-1)}{\mu^2}n(t)^{2/(2-h)}. 
$$
Here, $y_t$ measures the expected convergence rate and $n(t)$ is the step size function which we assume to be decreasing in $t$.

By using induction in $t$, we can solve the recursion as
$$
y_{t+1} 
\leq  \sum_{i=0}^t [\prod_{j=i+1}^t (1 - n(j)v(n(j)))] (2N+1)n(i)^2 + y_0\prod_{i=0}^t(1 - n(i)v(n(i))) .
$$
Since $1-x \leq \exp(-x)$ for all $x\geq 0$, 
$$\prod_{j=i+1}^t (1 - n(j)v(n(j)))\leq  \exp(-\sum_{j=i+1}^t n(j)v(n(j))).$$ 
Since $n(j)$ is decreasing in $j$ and $v(\eta)$ is increasing in $\eta$, $n(j)v(n(j))$ is decreasing in $j$ and we have
$$\sum_{j=i+1}^t n(j)v(n(j)) \geq \int_{x=i+1}^{t+1}n(x)v(n(x))dx.$$
Combining the inequalities above, we have
\begin{eqnarray*}
y_{t+1} &\leq& \sum_{i=0}^t \exp(-\sum_{j=i+1}^t n(j)v(n(j))) (2N+1)n(i)^2  + y_0\exp(-\sum_{j=0}^t n(j)v(n(i))) \\
&\leq& \sum_{i=0}^t \exp(-\int_{x=i+1}^{t+1}n(x)v(n(x))dx) (2N+1)n(i)^2  + y_0 \exp(-\int_{x=0}^{t+1}n(x)v(n(x))dx) \\
&=&\sum_{i=0}^t \exp(-[M(t+1)-M(i+1)]) (2N+1)n(i)^2 + \exp(-M(t+1)) y_0,
\end{eqnarray*}
where 
$$M(y)=\int_{x=0}^y n(x)v(n(x)) dx \mbox{ and } \frac{d}{dy}M(y) = n(y)v(n(y)).$$

We further analyze the sum in the above expression: 
\begin{eqnarray*}
S &=&\sum_{i=0}^t \exp(-[M(t+1)-M(i+1)]) (2N+1)n(i)^2  \\
%\\
%&=& \frac{4(2N+c(e)^{-1}-1)}{\mu^2} \sum_{i=1}^t \exp(-[M(t+1)-M(i+1)])n(i)^{2/(2-h)}.
%\end{eqnarray*}
%Let
%\begin{eqnarray*}
%F &=& \sum_{i=1}^t \exp(-[M(t+1)-M(i+1)])n(i)^{2/(2-h)}\\
&=&\exp(-M(t+1)) \sum_{i=0}^t \exp(M(i+1))(2N+1) n(i)^{2}.
\end{eqnarray*}
We know that $\exp(M(x+1))$ increases and $n(x)^{2}$ decreases, hence, in the most general case either their product first decreases and then starts to increase or their product keeps on increasing. We first discuss the decreasing and increasing case.
  Let $a(x)=\exp(M(x+1))n(x)^{2}$ denote this product and let integer $j\geq 0$ be such that $a(0)\geq a(1) \geq \ldots \geq a(j)$ and $a(j)\leq a(j+1)\leq a(j+2)\leq \ldots$ (notice that $j=0$ expresses the situation where $a(i)$ only increases). Function $a(x)$ for $x\geq 0$ is minimized for some value $h$ in $[j,j+1)$. For $1\leq i\leq j$, $a(i)\leq \int_{x=i-1}^{i} a(x) \mbox{d}x$, and for $j+1\leq i$, $a(i)\leq \int_{x=i}^{i+1} a(x) dx$.
 This yields the upper bound 
 \begin{align*}
 \sum_{i=0}^t a(i) &= a(0) + \sum_{i=1}^{j} a(i) + \sum_{i=j+1}^t a(i) \\
 &\leq a(0) + \int_{x=0}^{j} a(x) dx + \int_{x=j+1}^{t+1} a(x) dx, \\
 &\leq a(0) + \int_{x=0}^{t+1} a(x) dx.
 \end{align*}
 The same upper bound holds for the other case as well, i.e., if $a(i)$ is only decreasing.
 We conclude 
 $$S \leq (2N+1) \exp(-M(t+1)) [\exp(M(1))n(0)^{2} + \int_{x=0}^{j+1} \exp(M(x+1))n(x)^{2} dx ] .$$
 Combined with
  $$M(x+1)=\int_{y=0}^{x+1}n(y)n(v(y))dy\leq \int_{y=0}^{x}n(y)n(v(y))dy + n(x)=M(x) +n(x) $$
  we obtain
   \begin{align*}
   S &\leq (2N+1)\exp(-M(t+1)) [\exp(M(1))n(0)^{2} + \int_{x=0}^{t+1} \exp(M(x))n(x)^{2} exp(n(x)) dx ]\\
   &\leq (2N+1) \exp(-M(t+1)) [\exp(M(1))n(0)^{2} + \exp(n(0)) \int_{x=0}^{t+1} \exp(M(x))n(x)^{2}  dx ].
   \end{align*}
   This gives
   \begin{eqnarray*}
y_{t+1} &\leq& (2N+1) \exp(-M(t+1)) [\exp(M(1))n(0)^{2} \\
&& + \exp(n(0)) \int_{x=0}^{t+1} \exp(M(x))n(x)^{2}  dx ]   
 + y_0 \exp(-M(t+1)) \\
&=&  (2N+1) \exp(n(0)) C(t+1)  \\
&& + \exp(-M(t+1)) [(2N+1)\exp(M(1))n(0)^{2}  + y_0], %\label{eq:Yt_Ct_nt}
\end{eqnarray*}
%where the second inequality follows from $M(t+1)\geq M(t)$.
where
$$ C(t) = \exp(-M(t)) \int_{x=0}^{t} \exp(M(x))n(x)^{2}  dx .$$
\end{proof}

Notice that if $v(\eta)=c\cdot \eta$ for some constant $c$, then $C(t)=(1-\exp(-M(t)))/c$, which approaches $1/c$ rather than $0$ for $t\rightarrow \infty$. In the main text we already concluded that linear $v(\eta)$ do not contain any information.

%TO DO If $v(\eta)=\eta^2$, then $C(t)=const$ for any decreasing $n(t)$. Sure nothing really new!

%REPLACE $2$ by $2/q$ -- explain that this will help with canceling the second term later on in a balanced way when bounding $E_t$.

We want to minimize $C(t)$ by appropriately choosing the step size function $n(t)$.
We first compute the derivative
\begin{equation}C'(t)= - n(t)v(n(t))C(t) + n(t)^2
=n(t)v(n(t)) [ \frac{n(t)}{v(n(t))}- C(t)].\label{diffC}
\end{equation}
Notice that $C(t)$ is decreasing, i.e., $C'(t)<0$, if and only if 
\begin{equation} C(t) \geq \frac{n(t)}{v(n(t))}.\label{eqCond} \end{equation}
This shows that $C(t)$ can at best approach $\frac{n(t)}{v(n(t))}$. 
For example, 
\begin{equation}
    \bar{C}(t)=2 %\frac{2}{q}
    \frac{n(t)}{v(n(t))} \label{optC}
\end{equation}
would be close to optimal. 
%-- here, $q\in (0,1]$ will chosen later. It turns out that a smaller $q$ leads to a worse $C(t)$, however, this worse $C(t)$ can be realized by a smaller step size function and this will help us to upper bound $\mathbb{E_t}$ (right now, we only focus on upper bounding $\mathbb{Y_t}$).
Substituting (\ref{optC}) back into (\ref{diffC}) gives
$$ \bar{C}'(t) = - 
%\frac{2}{q}
n(t)^2.$$
Hence,
$$ n(t) = \sqrt{-\bar{C}'(t)}$$
and substituting this back into (\ref{optC}) gives the differential equation
$$ \bar{C}(t) = \frac{2 \sqrt{-\bar{C}'(t)} }{v(\sqrt{-\bar{C}'(t)})}.$$
%Let $\bar{C}(t)$  be its solution. 
For the corresponding step size function $n(t)$
we know that the actual $C(t)$ starts to behave like $\bar{C}(t)$ for large enough $t$, i.e., as soon as (\ref{eqCond}) is approached. So, after $C(0)=0$ (for $t=0$, the term $B\cdot \exp{-M(0)}$ will dominate), $C(t)$ increases until it crosses $\frac{n(t)}{v(n(t))}$ after which it starts decreasing and approaches $\bar{C}(t)$, the solution of the differential equation. We have
$$\frac{n(t)}{v(n(t))} = \bar{C}(t)/2 \leq C(t) \mbox{ for } t \mbox{ large enough}$$
and
$$C(t)\leq \bar{C}(t) \mbox{ for all } t\geq 0.$$
The above can be used to show that the actual $C(t)$ is at most a factor $2$ larger than the smallest $C(t)$ over all possible step size functions $n(t)$ for $t$ large enough.

%TO DO Be more accurate. Actually needs an argument.

\begin{lem}
A close to optimal step size can be computed by solving the differential equation 
$$ \bar{C}(t) = \frac{2 \sqrt{-\bar{C}'(t)} }{v(\sqrt{-\bar{C}'(t)})}$$
and equating
$$n(t)= \sqrt{-\bar{C}'(t)}.$$
The solution to the differential equation approaches $C(t)$ for $t$ large enough: For all $t\geq 0$, $C(t)\leq \bar{C}(t)$. For $t$ large enough, $C(t)\geq \bar{C}(t)/2$. 
\end{lem}

%CHOOSE $q=h$

%TO DO Main text Refer Lemma 4 and Cor 1 and use beta

We will solve the differential equation for
$$v(\eta) =\beta h \eta^{1-h}$$
with $h\in (0,1]$, where $\beta$ is a constant and $0\leq \eta \leq r$ for some $r\in (0,\infty]$ (including the possibility $r=\infty$).
%which we rewrite as $v(\eta)= \beta \eta^{1-h}$.
This gives the differential equation
$$\bar{C}(t) = \frac{2}{\beta h} \left(\sqrt{-\bar{C}'(t)}\right)^{h} = \frac{2}{\beta h} [-\bar{C}'(t)]^{h/2}.$$
We try the solution
$$\bar{C}(t) = c t^{-h/(2-h)}$$
for some constant $c$.
Plugging this into the differential equation gives
$$c t^{-h/(2-h)} = \frac{2}{\beta h} [\frac{ch}{(2-h)}t^{-h/(2-h)-1}]^{h/2}.  $$
Notice that $(-h/(2-h)-1)(h/2)= - (2/(2-h))(h/2)=-h/(2-h)$ so that the $t$ terms cancel.
We need to satisfy
$$ c = \frac{2}{\beta h} [\frac{ch}{(2-h)}]^{h/2}, $$
i.e., we must choose
$$ c = [h/(2-h)]^{h/(2-h)}  (2/\beta h)^{2/(2-h)} = [1/(2-h)]^{h/(2-h)}  (2/\beta )^{2/(2-h)}.$$

%Notice that, for $h\downto 0$, $c=2/\beta$ and $\bar{C}(t)= c =2/\beta$ and this gives a bad upper bound, yet, we will show that it can be used to prove a proper rate of convergence bound for the expectation of $E_t$ (rather than $Y_t$). 

For $n(t)$ we derive
\begin{align*}
n(t) &= \sqrt{-\bar{C}'(t)}
 = [(h/(2-h)) c t^{-h/(2-h)-1} ]^{1/2} = \sqrt{ch/(2-h)}t^{-1/(2-h)} \\ &=
 \left(\frac{2}{\beta (2-h)}\right)^{1/(2-h)} t^{-1/(2-h)}.    
\end{align*}
 
 We need to be careful about the initial condition: $\eta_0\leq \frac{1}{2L}$ and also $\eta_t\leq \eta_0\leq r$. To realize these conditions we make 
 $$\eta_0=\min\{\frac{1}{2L},r\} $$
 by defining $\eta_t = n(t+\Delta)$ for some suitable $\Delta$. Notice that by starting with the largest possible step size, $C(t)$ will cross $\frac{n(t)}{v(n(t))}$ as soon as possible so that it starts approaching the close to optimal $\bar{C}(t)$ as soon as possible.
 
 \begin{lem} %TO DO REWRITE
For $$v(\eta) =\beta h \eta^{1-h}$$
with $h\in (0,1]$, where $\beta>0$ is a constant and $0\leq \eta \leq r$ for some $r\in (0,\infty]$ (including the possibility $r=\infty$),
%with $$\beta = \frac{\mu}{2}  h^{-h}(1-h)^{-(1-h)}  r^h,$$
we obtain
$$\bar{C}(t) = [1/(2-h)]^{h/(2-h)}  (2/\beta)^{2/(2-h)}
%2 [h/(2-h)]^{h/(2-h)} / (\frac{\mu}{2}  h^{-h}(1-h)^{-(1-h)}  r^h)^{2/(2-h)}
(t+\frac{2\max\{2L,1/r\}}{\beta (2-h)})^{-h/(2-h)} $$
with
$$ n(t) =  \left(\frac{2}{\beta (2-h)}\right)^{1/(2-h)} (t+\frac{2\max\{2L,1/r\}}{\beta (2-h)})^{-1/(2-h)}.
%\sqrt{ [h/(2-h)]^{2/(2-h)} / (\frac{\mu}{2}  h^{-h}(1-h)^{-(1-h)}  r^h)^{2/(2-h)}} t^{-1/(2-h)}.
$$
\end{lem}

We will apply the lemma to two cases: $v(\eta)=\mu h \eta^{1-h}$ with $h=1/2$, and $v(\eta)=\frac{\mu}{2}h\eta^{1-h}$ for $h=1$. In both cases $r=\infty$.

We first consider $h=1/2$ for which $\beta=\mu$. This gives $\Delta=\frac{8L}{3\mu}$. We obtain
\begin{eqnarray*}
\bar{C}(t) &=& (2^5/3)^{1/3} \frac{1}{\mu^{4/3}} (t+\frac{8L}{3\mu})^{-1/3}, \\
n(t)&=&(\frac{4}{3\mu})^{2/3} (t+\frac{8L}{3\mu})^{-2/3}= (\frac{4}{3\mu t + 8L})^{2/3}, \\
A&=&(2N+1)e^{1/(2L)}. \\
%B&=&\frac{2N+1}{4L^2} MISS (1+ \frac{3\mu}{8L})^{2/3} + \mathbb{E}[Y_0].
\end{eqnarray*}
We can apply this to  $\mathbb{E}[Y_t]$ and $\frac{1}{t} \sum_{i=t+1}^{2t} \mathbb{E}[E_i]$ in Theorem \ref{thYE}:
$$\mathbb{E}[Y_{t}]\leq 
(2N+1)e^{1/(2L)} \frac{1}{\mu} \left( \frac{ 32 }{3\mu t+8L}\right)^{1/3} + O(t^{-2/3})$$
and
$$\frac{1}{t} \sum_{i=t+1}^{2t} \mathbb{E}[E_i] \leq
(2N+1)(e^{1/(2L)}+1)\frac{1}{\mu} \left(   \frac{2(6\mu t+8L)^{ 2}}{ (3\mu t+8L)t^3 } \right)^{1/3}+ O(t^{-1}). $$

Next we consider the case $h=1$ for which $\beta=\frac{\mu}{2}$. 
This gives $\Delta=\frac{8L}{\mu}$. We obtain
\begin{eqnarray*}
\bar{C}(t) &=& (4/\mu)^2  (t+\frac{8L}{\mu})^{-1}, \\
n(t)&=&\frac{4}{\mu} (t+\frac{8L}{\mu})^{-1}= \frac{4}{\mu t + 8L}, \\
A&=&(2N+1)e^{1/(2L)}. \\
%B&=&\frac{2N+1}{4L^2} (1+ \frac{3\mu}{8L})^{2/3} + \mathbb{E}[Y_0].
\end{eqnarray*}
We can apply this to  $\mathbb{E}[Y_t]$ and $\frac{1}{t} \sum_{i=t+1}^{2t} \mathbb{E}[E_i]$ in Theorem \ref{thYE}:
$$\mathbb{E}[Y_{t}] \leq (2N+1)e^{1/(2L)} \frac{1}{\mu} \frac{ 16 }{(\mu t+ 8L)} + O(t^{-2})$$
and
$$ \frac{1}{t} \sum_{i=t+1}^{2t} \mathbb{E}[E_i] \leq
(2N+1)(e^{1/(2L)}+1)\frac{1}{\mu}  \frac{4(2\mu t+8L)}{ (\mu t+8L)t } + O(t^{-2}).
$$

\section{Related Works}

In~\cite{gordji2016phi}, the authors define and study $\phi$-convex functions, and  $\phi_b$-convex and $\phi_E$-convex functions which are the generalization of $\phi$-convex functions. Indeed, $\phi$ is a mapping from $\mathbb{R}\times \mathbb{R}$ to $\mathbb{R}$. Hence, it is very different from our $\omega$ function, i.e., $\omega:[0,\infty) \rightarrow [0,\infty)$.  

In~\cite{AugerH13}, the authors discuss positive homogeneous functions. As defined in Definition 3.3, A function $F:\mathbb{R}^n \rightarrow \mathbb{R}$ is said \textit{positively homogeneous with degree $\alpha$} if for all $\rho>0$ and for all $w \in \mathbf{R}^n, F(\rho w) = \rho^\alpha F(w)$. It is obvious this class of functions is very different from our $F$ $\omega$-convex functions. 

In~\cite{csiba2017global}, the authors studied the convergence of class of functions which satisfy the following conditions
\begin{enumerate}
    \item \textbf{Strong} Polyak-Lojasiewics condition: 
      $$\frac{1}{2} \|\nabla F(w) \|^2 \geq \mu (F(w)-F(w_*)), \forall w \in \mathbb{R}^n.$$
     \item \textbf{Weak} Polyak-Lojasiewics condition:
        $$\|\nabla F(w) \| \|w-w_* \| \geq \sqrt{\mu} (F(w)-F(w_*)), \forall w \in \mathbb{R}^n.$$
\end{enumerate}
Moreover, they also consider the $\phi$-gradient dominated functions, i.e., $F:\mathbb{R}^n \rightarrow \mathbb{R}$ is $\phi$-gradient dominated if there exists a function $\phi:\mathbb{R}_+ \rightarrow \mathbb{R}_+$ such that $\phi(0)=0$, $\lim_{t\rightarrow 0} \phi(t)=0$ and 
$$
   F(w)-F(w_*) \leq \phi(\|\nabla F(w) \|), \forall w \in \mathbb{R}^n. 
$$

Compared to our $F$ $\omega$-convex functions, the studied class of functions $F$ in~\cite{csiba2017global} as introduced above is very different from the one studied in our paper. 

% In~\cite{AugerH13}, the authors discuss positive homogeneous functions. As defined in Definition 3.3, A function $F:\mathbb{R}^n \rightarrow \mathbb{R}$ is said \textit{positively homogeneous with degree $\alpha$} if for all $\rho>0$ and for all $\mathbf{x} \in \mathbf{R}^n, F(\rho \mathbf{x}) = \rho^\alpha F(\mathbf{x})$. It is obvious this class of functions is very different from our $F$ $\omega$-convex functions. 

% In~\cite{csiba2017global}, the authors studied the convergence of class of functions which satisfy the following conditions
% \begin{enumerate}
%     \item \textbf{Strong} Polyak-Lojasiewics condition: 
%       $$\frac{1}{2} \|\nabla F(\mathbf{x}) \|^2 \geq \mu (F(\mathbf{x})-F(\mathbf{x}_*)), \forall \mathbf{x} \in \mathbb{R}^n.$$
%      \item \textbf{Weak} Polyak-Lojasiewics condition:
%         $$\|\nabla F(\mathbf{x}) \| \|\mathbf{x}-\mathbf{x}_* \| \geq \sqrt{\mu} (F(\mathbf{x})-F(\mathbf{x}_*)), \forall \mathbf{x} \in \mathbb{R}^n.$$
% \end{enumerate}
% Moreover, they also consider the $\phi$-gradient dominated functions, i.e., $F:\mathbb{R}^n \rightarrow \mathbb{R}$ is $\phi$-gradient dominated if there exists a function $\phi:\mathbb{R}_+ \rightarrow \mathbb{R}_+$ such that $\phi(0)=0$, $\lim_{t\rightarrow 0} \phi(t)=0$ and 
% $$
%   F(\mathbf{x})-F(\mathbf{x}_*) \leq \phi(\|\nabla F(\mathbf{x}) \|), \forall \mathbf{x} \in \mathbb{R}^n. 
% $$

% Compared to our $F$ $\omega$-convex functions, the studied class of functions $F$ in~\cite{csiba2017global} as introduced above is very different from the one studied in our paper. 

\end{document}